\documentclass[a4paper,10pt,reqno]{amsart}

\usepackage{amsmath}
\usepackage{amsthm}
\usepackage{amsfonts}
\usepackage{amssymb}
\usepackage{mathrsfs}
\usepackage[shortlabels]{enumitem}
\usepackage{graphicx}
\usepackage{subfigure}
\usepackage{color}
\usepackage[dvipsnames]{xcolor}
\usepackage{setspace}
\usepackage[update,prepend]{epstopdf}
\usepackage{tikz}
\usepackage{mathtools} 
\makeatletter
\newlength{\bibsep}{\@listi \global\bibsep\itemsep \global\advance\bibsep by\parsep}
\makeatother

\newtheorem{theorem}{Theorem}[section]
\numberwithin{theorem}{section}
\newtheorem{prop}[theorem]{Proposition}
\newtheorem{lemma}[theorem]{Lemma}
\newtheorem{corollary}[theorem]{Corollary}
\theoremstyle{definition}
\newtheorem{definition}[theorem]{Definition}
\newtheorem{rem}[theorem]{Remark}

\numberwithin{equation}{section}

\newcommand{\N}{\mathbb{N}}

\newcommand{\R}{\mathbb{R}}

\newcommand{\C}{\mathbb{C}}

\newcommand\id{\operatorname{id}}

\newcommand\I{\mathrm{i}}
\newcommand\re{\operatorname{Re}}
\newcommand\im{\operatorname{Im}}

\newcommand\fa{\mathfrak{a}}

\newcommand\dom{\operatorname{dom}}
\newcommand\ran{\operatorname{ran}}
\DeclareMathOperator{\Div}{div}
\DeclareMathOperator{\curl}{curl}
\DeclareMathOperator{\essinf}{ess\,inf}
\newcommand\diag{{\mathrm{diag}}}

\newcommand\cS{\mathcal S}
\newcommand\cT{\mathcal T}

\newcommand\cA{\mathcal A}
\newcommand\cB{\mathcal B}
\newcommand\cC{\mathcal C}

\newcommand\cJ{\mathcal J}
\newcommand\cL{\mathcal L}

\newcommand\cH{\mathcal H}
\newcommand\cK{\mathcal K}

\newcommand\ft{\mathfrak t}
\newcommand\fs{\mathfrak s}

\newcommand\ov\overline

\newcommand\eps\varepsilon
\renewcommand\epsilon\varepsilon
\renewcommand\rho\varrho
\newcommand\al\alpha
\newcommand\la\lambda

\newcommand\ds\displaystyle

\newcommand\p\partial

\newcommand{\sto}{\stackrel{s}{\rightarrow}}

\newcommand{\gsr}{\stackrel{gsr}{\rightarrow}}

\newcommand{\dist}{\operatorname{dist}}

\newcommand{\supp}{\operatorname{supp}}

\newcommand{\beq}{\begin{equation}}
\newcommand{\eeq}{\end{equation}}
\newcommand{\be}{\begin{equation*}}
\newcommand{\ee}{\end{equation*}}
\newcommand{\bmat}{\begin{pmatrix}}
\newcommand{\emat}{\end{pmatrix}}

\colorlet{DarkOlive}{OliveGreen!70!Black}

\newcounter{counter_a}

\DeclarePairedDelimiter{\norma}{\lVert}{\rVert}

\newcommand{\diagdots}[3][-25]{%
  \rotatebox{#1}{\makebox[0pt]{\makebox[#2]{\xleaders\hbox{$\cdot$\hskip#3}\hfill\kern0pt}}}%
}

\author[F.~Ferraresso]{Francesco Ferraresso}
\address{School of Mathematics,
Cardiff University, Abacws,
Senghennydd Road, Cathays,
Cardiff CF24 4AG, UK}
\email{FerraressoF@cardiff.ac.uk}

\author[M.~ Marletta]{Marco Marletta}
\address{School of Mathematics,
Cardiff University, Abacws,
Senghennydd Road, Cathays,
Cardiff CF24 4AG, UK}
\email{MarlettaM@cardiff.ac.uk}

\date{\today}

\thanks{}
\usepackage[update,prepend]{epstopdf}

\title[Spectral properties of the dissipative Drude-Lorentz model]{Spectral properties of the inhomogeneous Drude-Lorentz model with dissipation}

\begin{document}

\begin{abstract}
We establish spectral enclosures and spectral approximation results for the  inhomogeneous lossy Drude-Lorentz system with purely imaginary poles, in a possibly unbounded Lipschitz domain of $\R^3$.
Under the assumption that the coefficients $\theta_e$, $\theta_m$ of the material are asymptotically constant at infinity, we prove that
spectral pollution due to domain truncation can lie only in the essential numerical range of a $\curl \curl_0 - f(\omega)$ pencil.\\
As an application, we consider a conducting metamaterial at the interface with the vacuum; we prove that the complex eigenvalues with non-trivial real part lie outside the set of spectral pollution. We believe this is the first result of enclosure of spectral pollution for the Drude-Lorentz model without assumptions of compactness on the resolvent of the underlying Maxwell operator.\\[0.2cm]
\emph{Classification:} 35P99 \: 35Q61 \: 47A56
\end{abstract}

%
%

\maketitle

\section{Introduction}
\subsection{The Drude-Lorentz model}
This paper concerns the spectra
and spectral approximation of a time-harmonic Drude-Lorentz model \cite{Tip} which commonly
occurs in the description of a class of metamaterials. This class includes doubly negative metamaterials,
which behave as if the electric permittivity and the magnetic permeability are simultaneously negative.  In the early 2000s,
it was conjectured that these materials might
allow the creation of a \emph{perfect lens} or an \emph{invisibility cloak}, see e.g. \cite{Nico}, \cite{Pendry}. Shortly afterwards,
experimental evidence of metamaterial cloaking at microwave frequencies \cite{SMJCPSS} and of optical superlensing \cite{Lee_2005}
was obtained. In the mathematics literature, `cloaking by anomalous localized resonances' has been intensively studied, see \emph{e.g.,} \cite{MR3035988}. Mathematically, some of the counter-intuitive spectral
properties of the time-dependent Maxwell system for an interface between a metamaterial and
a vacuum are investigated for the non-dissipative case, in a special geometry, in \cite{MR3764925, CHJII}. For the dissipative case, in the whole space and in a setting allowing dimension-reduction, we refer to the recent article \cite{BDPW}.
Here we consider a more general Drude-Lorentz system
\begin{equation}\label{PDEs}
\begin{split}
&\curl \hat{H} =  i \omega\bigg( 1  - \frac{(\theta_e)^2}{\omega^2 + i \gamma_e \omega} \bigg) \hat{E},
\;\; -\curl\hat{E} =  i \omega\bigg(1  - \frac{(\theta_m)^2}{\omega^2 + i \gamma_m \omega} \bigg) \hat{H}, \\
&(\nu \times \hat{E})|_{\p \Omega} = 0,
\end{split}
\end{equation}
in a bounded or unbounded Lipschitz domain $\Omega\subseteq {\mathbb R}^3$ with outer normal $\nu$. The variable $\omega$ is the spectral parameter and $\theta_e$, $\theta_m$ are bounded real-valued functions.  We describe the essential
spectrum and its decomposition into parts connected with the behaviour of the coefficients at infinity and parts due to local dissipative effects. We obtain tight a-priori
enclosures for the set in which these different components of the spectrum may lie.
Adapting new non-selfadjoint techniques from \cite{MR3694623} and \cite{BFMT} to the setting
of meromorphic operator-valued functions, we examine how the spectrum behaves under perturbation of the
domain $\Omega$. We obtain unexpectedly small enclosures for the sets where spectral pollution
\cite[Def. 2.2]{MR3694623} may appear if an unbounded $\Omega$ is approximated by a large, bounded $\Omega$.

We now describe the problem in more detail. Starting from Maxwell's equations
\[
\partial_t D =  \curl H,  \quad \partial_t B = - \curl E, \quad  \Div D = 0, \quad \Div B = 0,\\
\]
relations between $(D, B)$ and $(E, H)$ must be imposed to capture the properties of the medium under consideration, see \cite{MR3023383} for an interesting discussion on the diverse constitutive relations and applications to linear bianisotropic media. The Drude-Lorentz model assumes these relations to be given by convolutions
\[
D(x,t) = E(x,t) + \hspace{-1mm}\int_{t_0}^t \hspace{-2mm}\chi_e(x, t-s) E(x, s) ds, \;\;
 B(x,t) = H(x,t) + \hspace{-1mm}\int_{t_0}^t \hspace{-2mm}\chi_m(x, t-s) H(x, s) ds.
\]
The functions $\chi_e(\cdot,t)$ and $\chi_m(\cdot,t)$ are assumed to be zero for $t<0$ and are usually described in terms of their Fourier transforms in time; for instance,
\[
\hat{\chi}_e(\omega) = - \frac{(\theta_e)^2}{\omega^2 + i \gamma_e \omega} - \sum_{n=1}^\infty \frac{(\Omega_n^e)^2}{\omega^2 +
i \gamma_n^e \omega - (\la_n^e)^2},
\]
in which $\la_n^e > 0$, $\Omega_n^e \geq 0$, $\gamma_e$, $\gamma_n^e > 0$, $n \in \N \cup \{0\}$, are constants, and
$\theta_e$ is some non-negative function.
From the equation $\curl \hat{H} = i \omega \hat{D} = i \omega(1 + \hat{\chi}_e) \hat{E}$ one then obtains
\[
\curl \hat{H} =  i \omega\bigg( 1  - \frac{(\theta_e)^2}{\omega^2 + i \gamma_e \omega} -
\sum_{n=1}^\infty \frac{(\Omega_n^e)^2}{\omega^2 + i \gamma_n^e \omega - (\la_n^e)^2}\bigg) \hat{E},
\]
together with a corresponding equation for $\curl\hat{E}$. In this paper, as in \cite{MR3764925}, we treat
the simplest case, namely the lossy Drude system \cite[\S6]{MR3421776} defined in \eqref{PDEs}.
\subsection{Notation} \label{notation}
\noindent $\bullet$ Let $\Omega \subset \R^3$ be an open set. $L^2(\Omega)^3 = L^2(\Omega, \C^3)$ is the standard Hilbert space of complex-valued vector fields having finite $L^2$-norm. The $L^2$-norm will be denoted by $\norma{\cdot}$.\\
$\bullet$ The homogeneous Sobolev or Beppo Levi space $\dot{H}^1(\Omega)$ is defined as the completion of $C^\infty_c(\overline{\Omega})^3$ with respect to the seminorm $\norma{u}_{\dot{H}^1(\Omega)} = \norma{\nabla u}$. \\
$\bullet$ $\nabla \dot{H}^1(\Omega) = \{ \nabla \varphi \in L^2(\Omega)^3 : \varphi \in \dot{H}^1(\Omega) \}$ will be regarded as a subspace of $L^2(\Omega)^3$. \\
$\bullet$ $H(\curl, \Omega) = \{ u \in L^2(\Omega)^3 : \curl u \in L^2(\Omega)^3 \}$ is endowed with the norm given by $\norma{u}^2_{H(\curl, \Omega)} =  \norma{u}^2 + \norma{\curl u}^2$.\\
$\bullet$ $H_0(\curl, \Omega)$ is the closure of $C^\infty_c(\Omega)^3$ with respect to $\norma{\cdot}_{H(\curl, \Omega)}$. If $\p \Omega$ is sufficiently regular, it can also be described as
\[
H_0(\curl, \Omega) = \{ u \in L^2(\Omega)^3 : \curl u \in L^2(\Omega)^3, \, \nu \times u = 0 \,\, \textup{on $\p \Omega$} \}
\]
$\bullet$ The differential expression $\curl$ is associated with two self-adjoint realisations in $L^2(\Omega)^3$. $\curl$ is the maximal one, with domain $\dom(\curl) = H(\curl, \Omega)$; $\curl_0$ the minimal one with domain $\dom(\curl_0) = H_0(\curl, \Omega)$. Note that $(\curl_0)^* = \curl$. \\
$\bullet$ $H(\Div, \Omega) = \{ u \in L^2(\Omega)^3 : \Div u \in L^2(\Omega) \}$, $\norma{u}^2_{H(\Div, \Omega)} = \norma{u}^2 + \norma{\Div u}^2$. \\
$\bullet$ $H(\Div 0, \Omega)$ is the subspace of $L^2(\Omega)^3$ of vector fields with null (distributional) divergence.\\
$\bullet$ Given a linear operator $T : \cH \supset \dom(T)\to \cH$,
\begin{align*}
&\sigma(T) = \{ \omega \in \C \,:\, T - \omega \:\: \textup{is not boundedly invertible} \}, \\
&\sigma_{\rm app}(T) = \{ \omega \in \C \,:\, \exists (u_n)_n \subset \dom(T),\, \norma{u_n} = 1,\, \norma{(T - \omega)u_n} \to 0 \} \\
&\sigma_e(T) := \{ \omega \in \C \,:\, \exists (u_n)_n \subset \dom(T),\, \norma{u_n} = 1,\, u_n \rightharpoonup 0,\, \norma{(T - \omega)u_n} \to 0 \}.
\end{align*}
For non-selfadjoint operators in complex Banach spaces, there are 5 non-equivalent definitions of essential spectrum, see \cite[Chp.9, p.414]{EE}, denoted by $\sigma_{ek}(T)$, $k = 1,\dots, 5$. Note that $\sigma_e(T) := \sigma_{e2}(T)$.\\
$\bullet$ Let $D \subset \C$ be a domain. Given $\omega \mapsto \cL(\omega)$, $\omega \in D$, a holomorphic family of closed linear operators with the same $\omega$-independent domain $\dom(\cL) = \dom(\cL(\omega))$, $\omega \in D$, we define $\sigma(\cL) = \{ \omega \in D : 0 \in \sigma(\cL(\omega)) \}$, and similarly we define point, continuous, residual, essential spectrum by replacing $\sigma$ with $\sigma_x$, $x = p, c, r, e$ in the previous formula.

\subsection{Operator formulations and main results}
The system (\ref{PDEs}) has several operator formulations, which we now outline. The equations hold in a (bounded or
unbounded) Lipschitz domain $\Omega\subseteq {\mathbb R}^3$, in which the functions
$\theta_e$ and $\theta_m$  are assumed to be bounded and non-negative. The Fourier transform $\hat{E}$ of the
electric field $E$ is supposed to lie in the space $H_0(\curl, \Omega)$, which encodes the boundary condition $\nu \times \hat{E} = 0$ on $\p \Omega$, while $\hat{H}$ is assumed
to lie in $H(\curl, \Omega)$. The first operator formulation of \eqref{PDEs} is then
\[  \cL(\omega) \binom{\hat{E}}{\hat{H}} = {\bf 0}, \]
in which $\omega\mapsto \cL(\omega)$ is the $2\times 2$ rational block-matrix pencil given by
\begin{equation}\label{eq:intro1}
\cL(\omega) = \begin{pmatrix}
- \omega + \frac{\theta_e^2}{\omega + i \gamma_e} & i \curl \\
-i\curl_0 & - \omega + \frac{\theta_m^2}{\omega + i \gamma_m}
\end{pmatrix}, \quad \dom(\cL) = H_0(\curl, \Omega) \oplus H(\curl, \Omega);
\end{equation}
see subsection \ref{notation} for definitions of the Sobolev spaces, $\curl$, $\curl_0$, etc.
It is not difficult to show (see \cite{MR3543766}) that the Drude-Lorentz pencil $\cL(\omega)$ is the first Schur
complement of the `companion' block operator matrix
\begin{equation}\label{def:cA}
\cA =
\begin{pmatrix}
A & B \\
B|_{\dom(A)} & -iD
\end{pmatrix}
\end{equation}
in $L^2(\Omega)^6$, with domain $\dom(\cA) = H_0(\curl, \Omega) \oplus H(\curl, \Omega) \oplus L^2(\Omega)^3 \oplus L^2(\Omega)^3$ and
\begin{equation}\label{ABDdef}
A =
\begin{pmatrix}
0 & i \curl \\
-i \curl_0 & 0
\end{pmatrix}, \quad
B = \begin{pmatrix}
\theta_e & 0 \\
0 & \theta_m
\end{pmatrix}, \quad
D = \begin{pmatrix}
\gamma_e & 0 \\
0 & \gamma_m
\end{pmatrix};
\end{equation}
in other words,
\begin{equation} \label{intro:cL}
\cL(\omega) = A - \omega - B (-iD - \omega)^{-1} B.
\end{equation}
In particular the spectrum of $\cA$ coincides with the spectrum of $\cL$ outside the two poles $-i \gamma_e, -i \gamma_m$. We will exploit this connection and the results in \cite{BFMT} to further decompose the spectrum of $\cL$ into the spectra of two operator pencils.
This method allows us to generalise the known spectral analysis of the Drude-Lorentz model in the following ways:\\
(1) In our assumptions, $0 \in \sigma_e(A)$, where $A$ is defined as in \eqref{ABDdef}, since $\nabla \dot{H}^1_0(\Omega)\oplus\nabla \dot{H}^1(\Omega)$ is an infinite-dimensional kernel of $A$. This is not allowed by many results in the literature, e.g. \cite[Proposition 2.2]{MR3543766}, where it is required that $A$ have compact resolvent. \\
(2) We allow the domain $\Omega$ to be unbounded. Consequently, contributions to $\sigma_e(\cA)$ are expected from infinity. \\
(3) We allow the coefficients $\theta_e$, $\theta_m$ to be \textit{both} non-constant, even though we assume that they are asymptotically constant. \\[0.1cm]
On the other hand, to avoid very singular situations we restrict ourselves to the case where $\gamma_e$ and $\gamma_m$ (namely, the position of the poles) are fixed.\\
A large part of the spectral analysis has been achieved not by inspecting directly the operator pencil $\cL$, but its first Schur complement $\cS_1$, defined on $\dom(\cS_1) := \{u \in H_0(\curl, \Omega): \Theta_m(\omega)^{-1} \curl_0 u \in H(\curl, \Omega)\}$ by
\begin{equation}
\label{S1}
\cS_1(\omega) = \curl \Theta_m(\omega)^{-1} \curl_0 - \frac{\Theta_e(\omega)}{(\omega + i \gamma_e)(\omega + i \gamma_m)},
\end{equation}
for $\omega \in \C \setminus (\{-i\gamma_e,-i\gamma_m\}\cup\overline{W(\Theta_m)})$; here  the notation $W(\cdot)$ denotes the numerical range of an operator or a pencil, see Def. \ref{def:W}, and
\begin{equation}\label{Thetadef}
\Theta_e(\omega) :=  \omega^2 + \omega i \gamma_e - \theta_e^2, \quad \Theta_m(\omega) := \omega^2 + \omega i \gamma_m - \theta_m^2,
\end{equation}
are $\omega$-quadratic multiplication pencils.
An important technical point is realising that $\cS_1(\omega)$, in general, cannot be defined either as an $m$-accretive operator or as a self-adjoint operator 
\emph{independently of $\omega \in \C \setminus \{-i \gamma_e, - i \gamma_m\}$}.
We overcome this obstacle by introducing a set
\begin{equation}
\label{def:Sigma}
\Sigma = \{\omega\in\C \, | \, \re(\omega)\im(\omega+i\gamma_m/2)\neq 0\},
\end{equation}
decomposing $\Sigma$ as a disjoint union $\Sigma = \Sigma_1 \dot{\cup}\Sigma_2$, and defining $\cS_1(\omega)$ in two different ways, depending on whether $\omega \in \Sigma_1$ or in $\Sigma_2$. In fact, $i\cS_1$ is $m$-accretive for $\omega \in \Sigma_1 \subset \Sigma$, while it is $m$-dissipative for $\omega \in \Sigma_2 = \Sigma \setminus \overline{\Sigma_1}$.\\
Note also that the relation between the spectrum of the operator pencil $\cL$ and that of $\cS_1$ is completely non-trivial. This is a frequently faced problem in the study of spectra of metamaterials where the dependence on the spectral parameter is non-linear, see e.g.\cite{MR4191388} where similar hurdles were encountered in the study of the essential spectrum of a negative metamaterial in a bounded domain. From our perspective, these difficulties are natural consequences of the lack of a \emph{diagonal dominance} pattern (in the sense of \cite[Def. 2.2.1]{TreB}) for the block operator matrices involved. In \cite[p.1187]{MR4191388} the operator matrix is upper-dominant. In our case, $\cL(\omega)$ is off-diagonally dominant, since the off-diagonal entries are differential operators of order 1 while the diagonal entries are of order 0. Unfortunately the off-diagonal entries are not boundedly invertible, so
standard theorems relating the spectrum of an operator matrix and that of its Schur complements, such as \cite[Thm. 2.3.3]{TreB}, do not apply. We overcome these difficulties by defining the Schur complement $\cS_1$ locally and by improving the abstract result \cite[Prop. 2.10.1(c)]{TreB}, which would allow only bounded and self-adjoint entries on the main diagonal. We note in passing that the question raised in \cite[p.1187]{MR4191388} can be partially solved by applying \cite[Prop. 2.10.1(b)]{TreB}.

Our first main result is a decomposition of the essential spectrum, see Proposition \ref{sigmaLS} and Theorem
\ref{sigma-ess}, which is summarised in the following theorem.

\begin{theorem}\label{thmintro}
Suppose that $\Omega$ is an unbounded Lipschitz open set of $\R^3$ and $\theta_e$, $\theta_m$ are asymptotically constant. Let $P_\nabla$ be the orthogonal projection from $ L^2(\Omega)^3\!=\!\nabla \dot H^1_0(\Omega) \oplus H(\Div 0,\Omega)$ onto $\nabla \dot{H}^1_0(\Omega)$. Let
\begin{equation} \label{eq: G}
G(\omega) = -P_{\nabla} \left( \frac{\Theta_e(\omega)}{(\omega + i \gamma_e)(\omega + i \gamma_m)} \right) P_{\nabla}
\end{equation}
Let $\cS_\infty$ be the pencil $\cS_1$ restricted to divergence-free vector fields and with coefficients $\theta_e, \theta_m$ constantly equal to their value at infinity. Finally, let $\Sigma$ be as in \eqref{def:Sigma}. Then, with $\sigma_{ek}$ denoting the essential spectrum as in \cite[Chp. IX, p.414]{EE},
\begin{equation} \label{intro:dec}
 \sigma_{ek}(\cL) \cap \Sigma = \sigma_{ek}(\cS_1) \cap \Sigma = (\sigma_{ek}(\cS_\infty) \cup \sigma_{ek}(G)) \cap \Sigma, \quad k=1,2,3,4,
\end{equation}
where $\sigma_{ek}(\cS_\infty)$ is described in Prop. \ref{prop: spectrum infty} and
\[
\sigma_{ek}(G) \subset
\begin{cases}- i [0, \gamma_e), \quad &\textup{if $- \frac{\gamma_e^2}{4} + \norma{\theta^2_e}_{\infty} \leq 0$,} \\
- i [0, \gamma_e) \cup \left(-d_e -i \frac{\gamma_e}{2}, d_e -i \frac{\gamma_e}{2} \right) \quad &\textup{if $- \frac{\gamma_e^2}{4} + \norma{\theta^2_e}_{\infty} > 0$.}
\end{cases}
\]
with $d_e \in \left(- \sqrt{- \frac{\gamma_e^2}{4} + \norma{\theta^2_e}_{\infty}}, \sqrt{- \frac{\gamma_e^2}{4} + \norma{\theta^2_e}_{\infty}} \right)$.
\end{theorem}

\begin{rem} In the case where $\Omega$ is bounded, decomposition \eqref{intro:dec} continues to hold. In such a case $\cS_\infty$ can be defined by freezing the coefficients $\theta_e$, $\theta_m$ at arbitrary positive values $\theta^0_e$, $\theta_m^0$; in fact, $\sigma_e(\cS_\infty) = \emptyset$ independently of the chosen values $\theta^0_e$, $\theta_m^0$. This can be proved along the lines of \cite[Theorem 6]{MR3942228}. Therefore, in bounded domains the only contribution to the essential spectrum comes from $\sigma_{e}(G)$. See also Remark \ref{rem:bdd_dec}.
\end{rem}

The explicit computation of $\sigma_{ek}(G)$ generally depends upon the regularity of the function $\theta_e$. If it is continuous, we obtain the equality
\[\sigma_{ek}(G) = \ran\biggl(- \frac{i \gamma_e}{2} - \sqrt{- \frac{\gamma_e^2}{4} + \theta_e^2}\biggr) \cup \,\ran\biggl(- \frac{i \gamma_e}{2} + \sqrt{- \frac{\gamma_e^2}{4} + \theta_e^2}\biggr).\]
Otherwise, the essential spectrum depends on the geometry of the set of discontinuities of $\theta_e$. If $\theta_e$ is a step function, some computations can be found in Section \ref{sec:example} below, which are based on the analytic results for transmission problems of Ola \cite{MR1362712} and Pankrashkin \cite{MR4041099}, initially investigated in the seminal paper \cite{MR782799}. In the example of Section \ref{sec:example}, $\sigma_{ek}(G)$  consists of at most six distinct points. For more complicated examples where the discontinuity interfaces are allowed to have non-convex corners, bands of essential spectrum can be generated, see \cite{MR1680769}. These problems have been tackled recently using the $\tt{T}$-coercivity method, see \cite{MR3200087}.

Our second fundamental result concerns spectral approximation of $\cL$ by the truncation method. This involves replacing $\Omega$ with a bounded domain $\Omega_n \subset \Omega$, and $\cL$ with $\cL_n$, which will be associated with Problem \eqref{eq:intro1} with the same \emph{electric boundary conditions}; when $n \to \infty$, $\Omega_n$ monotonically increases and covers the whole of $\Omega$. The question is whether $\sigma(\cL_n)$ will be `close to' $\sigma(\cL)$ as $n \to \infty$. This is already an interesting problem for self-adjoint operators having band-gap spectrum \cite{MR2640293}: indeed, the gaps in the essential spectrum may contain eigenvalues of the approximating operators as $n \to \infty$, or equivalently, the spectral gaps may contain points of the \emph{spectral pollution} set, given by
\[
\sigma_{\rm poll}((\cL_n)_n) = \{ \omega \in \rho(\cL) \,:\, \exists I \subset \N, \: \textup{$I$ infinite},\: \omega_n \in \sigma(\cL_n),\,\, \omega_n \to \omega, \, n \in I,\,  n \to \infty \}.
\]
For the Maxwell pencil $\cL$ defined in \eqref{intro:cL}, the presence of spectral pollution for the approximating sequence $(\cL_n)_n$ is almost inevitable, since $\cL_n$ is a non-self-adjoint, rational pencil of operators for every $n$, having non-trivial essential spectrum even in bounded domains. It is then of pivotal importance to determine where spectral pollution may appear; and, on the other hand, which spectral points $\omega \in \sigma(\cL)$ can be approximated exactly via domain truncation.
Theorem \ref{thm: final} shows that spectral pollution for the sequence $\cL_n$, $n \in \N$ can only occur in the {\it essential numerical range} $W_e(\cS_\infty)$ of the constant coefficient pencil $\cS_\infty$ given on a suitable domain by 
\begin{equation}\label{s1inf}
\cS_{\infty}(\omega) = \Theta_{m, \infty}(\omega)^{-1}\curl\curl_0 - \frac{\Theta_{e, \infty}(\omega)}{(i\gamma_e + \omega)(i \gamma_m + \omega)},
\end{equation}
in which
$\Theta_{e,\infty}(\Omega) = \omega^2 + i \gamma_e \omega - (\theta_e^{0})^2$,
$\Theta_{m,\infty}(\Omega) = \omega^2 + i \gamma_m \omega - (\theta_m^{0})^2$,
and $\theta_e^0$, $\theta_m^0$ are the values of $\theta_e$ and $\theta_m$ at infinity. We recall that
\[
\begin{split}
W_e(\cS_\infty)= \{ &\omega \in \C : \,\exists u_n \in \dom(\cS_\infty(\omega)),\,n \in \N, \\
&\norma{u_n} = 1,\: u_n \rightharpoonup 0, (\cS_\infty(\omega) u_n, u_n) \to 0, \, n \to \infty \}.
\end{split}
\]
In particular, the set of spectral pollution is always contained in the union of one-dimensional curves in $\C$, improving in a substantial way abstract enclosures for the spectral pollution set in term of the essential numerical range of $\cL$, which in this case would establish only that spectral pollution is contained in the infinite horizontal strip $\R \times [-\gamma_e, 0]$.

The structure of this article is as follows. In Section \ref{sec:numr} we first establish a basic numerical range enclosure for the whole of $\sigma(\cL)$, see Prop. \ref{prop:numrenc}; we then define the operator pencil $\cS_1$ in a rigorous way in Thm \ref{thm: s_1 form_refined} and we prove that the spectral properties of $\cL$ are retained by $\cS_1$ inside $\Sigma$, see Prop. \ref{sigmaLS}. This result is then exploited to prove a refined numerical range enclosure, see Thm. \ref{thm:refnumran}. In Section \ref{sec:ess_spec} we prove Thm. \ref{sigma-ess}, which establishes that $\sigma_e(\cS_1) \cap \Sigma$ can be decomposed in the union of the essential spectra of two operator pencils, $\cS_\infty$ capturing the behaviour at infinity due to divergence-free vector fields; $G$ capturing the contribution of gradients. Section \ref{sec:red_op} contain qualitative results regarding the essential spectra of the reduced operators $\cS_\infty$ and $G$. In Section \ref{sec:lim_ess_spec} we then prove Thm. \ref{thm: final}, establishing that spectral pollution for the domain truncation method is contained in $W_e(\cS_\infty)$, and an approximation property for the isolated eigenvalues of $\cL$. Finally, Section \ref{sec:example} contains explicit computations for the case of locally constant functions $\theta_e(x) = \alpha_e \chi_K(x)$, $\theta_m(x) = \alpha_m \chi_K(x)$, which are identically zero at infinity.

\section{Numerical range, Schur complements and spectral enclosures}
\label{sec:numr}

Let $\cH = L^2(\Omega)^3\oplus L^2(\Omega)^3$.
The operators $A$, $B$ and $D$ appearing
in the definition (\ref{def:cA}) of $\cA$ have domains $\dom(A) = H_0(\curl, \Omega) \oplus H(\curl, \Omega)$, $\dom(B) = \dom(D) = \cH$; the fact that
$\dom(B)=\dom(D)=\cH$ relies on our assumption that the functions $\theta_e$ and $\theta_m$ lie in $L^\infty(\Omega, \R)$.
Since the off-diagonal operators $B$ and $D$ are bounded,
$\cA$ is a diagonally dominant, closed $\cJ$-self-adjoint operator matrix, where $\cJ = \diag(i,-i, i, -i) J$, and $J$ is the standard componentwise complex conjugation. In particular, $\sigma(\cA) = \sigma_{app}(\cA)$. Due to \cite[Thm 2.3.3 (ii)]{TreB}, $\sigma(\cA) \setminus \sigma(-iD) = \sigma(\cL)$, with equality for the point, continuous, and essential spectrum as well.

\begin{prop} \label{prop: num range 1}
Let $M:= \max\{\gamma_e,\gamma_m\}$. Then the numerical range $W(\cA)$ of the block operator matrix $\cA$ is contained in the
strip $\R \times [-iM,0]$.
\end{prop}
\begin{proof}
If $\omega \in W(\cA)$, by definition there exists $(u,v) \in \cH$, $\norma{u}^2 + \norma{v}^2 = 1$ such that
\[
(Au,u) + 2 \re (Bv,u) - i (Dv,v) = \omega.
\]
Hence, $0 \geq \im \omega = - \re (Dv, v) = - \gamma_e \norma{v_1}^2 - \gamma_m \norma{v_2}^2 \geq - \max\{\gamma_e, \gamma_m\}$.
\end{proof}

\begin{rem}
The previous enclosure holds independently on the domain $\Omega$ and it holds for non-constant, positive, and bounded $\gamma_e$ and $\gamma_m$ by replacing them with $\norma{\gamma_e}_{\infty}$ and $\norma{\gamma_m}_{\infty}$.
\end{rem}

On the other hand, $\C \setminus \{-i \gamma_e, -i \gamma_m\} \ni \omega \mapsto \cL(\omega)$ defines a pencil of block operator matrices in $\cH$, given explicitly by
\begin{equation} \label{def: cL}
\cL(\omega) = \begin{pmatrix}
-\omega + \frac{\theta_e^2}{(\omega + i \gamma_e)} & i \curl \\
-i\curl_0 & -\omega + \frac{\theta_m^2}{(\omega + i \gamma_m)}
\end{pmatrix}
\end{equation}
where $\dom(\cL(\omega)) = H_0(\curl, \Omega) \oplus H(\curl, \Omega)$.

\begin{definition} \label{def:W}
Given a linear operator $T$ with domain $\dom(T) \subset H$ on a Hilbert space $H$, the numerical range of $T$ is
\[
W(T) = \{ \omega \in \C :\, \exists u \in \dom(T), \, \norma{u} = 1, \, (T u, u) = \omega \}.
\]
Let $\cT$ be a pencil of linear operators $\cT(\omega)$ with $\dom(\cT(\omega)) \subset H$. We define
\[
W(\cT) = \{ \omega \in \C :\, 0 \in \ov{W(\cT(\omega))} \}.
\]
\end{definition}

\begin{rem}
Note that $W(\cT)$ is denoted by $W_{\Psi}(\cT)$ in the recent article \cite[Equation (1.2)]{MR4447382}, \emph{cf.} \cite{BM}. The main reason to use this set in place of the classical one $\widetilde{W}(\cT) = \{ \omega \in \C : 0 \in W(\cT(\omega)) \}$ is that in general $\sigma_{\rm app}(\cT)$ is not a subset of the closure of $\widetilde{W}(\cT)$. Instead, it is immediate to check that $\sigma_{\rm app}(\cT) \subset W(\cT)$.
\end{rem}

\begin{prop}
\label{prop:numrenc} 
The numerical range $W(\cL)$ is contained in the non-convex
 subset of $\mathbb C$ described by the inequality
 \begin{equation}\label{enc} 0 \leq -\im\omega \leq \min\left(M, \frac{\gamma_e \| \theta_e \|_\infty^2 + \gamma_m \|\theta_m\|_\infty^2}{(\re\omega)^2} \right),
 \end{equation}
  in which $M=\max\{\gamma_e,\gamma_m\}$.

\end{prop}
\begin{proof}
Let $\omega \in W(\cL)$; then by definition $0 \in \ov{W(\cL(\omega))}$. Let $(u_n,v_n) \in H_0(\curl, \Omega) \oplus H(\curl, \Omega)$, $\norma{u_n}^2 + \norma{v_n}^2 = 1$, $n \in \N$ and consider the equation
\[
- \omega + \left(\frac{\theta_e^2}{(\omega + i \gamma_e)}u_n, u_n\right) + 2 \re (i\curl v_n, u_n) + \bigg(\frac{\theta_m^2}{(\omega + i \gamma_m)} v_n, v_n\bigg) = \eps_n
\]
with $\eps_n \in \C$, $\eps_n \to 0$ as $n \to \infty$. Upon taking the imaginary part we see that
\begin{equation}\label{imp}
- \im \omega - \bigg( \frac{(\im \omega + \gamma_e)}{|\omega + i \gamma_e|^2} \theta_e^2 u_n, u_n \bigg) - \bigg( \frac{(\im \omega + \gamma_m)}{|\omega + i \gamma_m|^2} \theta_m^2 v_n, v_n \bigg) = \im \eps_n.
\end{equation}
If $\im \omega > 0$, then $\im (\omega + \eps_n) \leq 0$; this gives a contradiction for $n \to \infty$, hence $\im \omega \leq 0$. Similarly, it cannot happen that both $\im \omega + \gamma_e < 0$ and $\im \omega + \gamma_m < 0$; hence $- \max\{ \gamma_e, \gamma_m \} \leq \im \omega \leq 0$, i.e. $-M \leq \im\omega \leq 0.$ To obtain (\ref{enc}), we observe that (\ref{imp}) may be rewritten as
\[
\begin{split}
 (-\im\omega)& \left\{1 + \frac{(\theta_e^2 u_n,u_n)}{|\omega+i\gamma_e|^2} + \frac{(\theta_m^2 v_n,v_n)}{|\omega+i\gamma_m|^2} \right\} \\
&= \frac{\gamma_e(\theta_e^2 u_n,u_n)}{(\re\omega)^2 + (\gamma_e+\im\omega)^2} + \frac{\gamma_m(\theta_m^2 v_n,v_n)}{(\re\omega)^2 + (\gamma_m+\im\omega)^2} + \im \eps_n \\
&\leq \,  \frac{\gamma_e \| \theta_e \|_\infty^2 + \gamma_m \|\theta_m\|_\infty^2}{(\re\omega)^2} + \im \eps_n
\end{split}
\]
The factor in parentheses $\left\{ \cdot \right\}$ on the left hand side exceeds $1$, so the result follows by taking the limit as $n \to \infty$.
\end{proof}

In order to make further progress we use an additional Schur complement argument on the pencil $\cL(\omega)$, which can be considered as a block operator matrix in $\cH = L^2(\Omega)^3 \oplus L^2(\Omega)^3$. In terms of the bounded quadratic pencils $\Theta_e$ and $\Theta_m$, see \eqref{Thetadef},
$\cL(\omega)$ has the form
\[
\cL(\omega) = \begin{pmatrix}
- \frac{\Theta_e(\omega)}{(i \gamma_e + \omega)} & i \curl \\
-i \curl_0 & - \frac{\Theta_m(\omega)}{(i \gamma_m + \omega)}
\end{pmatrix}.
\]
If $\omega$ is an eigenvalue of $\cL$ with eigenfunction $\binom{E}{H}$, then
\[
\begin{cases}
- \frac{\Theta_e(\omega)}{(i \gamma_e + \omega)}E + i \curl H = 0, \\
-i \curl_0 E - \frac{\Theta_m(\omega)}{(i \gamma_m + \omega)}H = 0.
\end{cases}
\]
Assume that $\Theta_m(\omega)$ is boundedly invertible. Formally, we could apply $\Theta_m(\omega)^{-1}$ from the left in the second equation and apply $\curl$; by using the first equation we obtain
\[
\curl \Theta_m(\omega)^{-1} \curl_0 E - \frac{\Theta_e(\omega)}{(i \gamma_e + \omega)(i \gamma_m + \omega)} E = 0
\]
for all $\omega \in \rho(\Theta_m)$. However, without further restrictions on $\omega$ the operator $\curl \Theta_m(\omega)^{-1} \curl_0$
 may not even be accretive.


\begin{theorem}\label{thm: s_1 form_refined}
Let $\Sigma_1$ be the set
\[ \Sigma_1 := \{ \omega\in{\mathbb C} \, | \, \re(\omega)\im(\omega+i\gamma_m/2)>0\}. \]
Then for each $\omega\in\Sigma_1$, the sesquilinear form
\[
\fs_1(\omega)[x,y] = (\Theta_m(\omega)^{-1} \curl_0 x, \curl_0 y) - \bigg( \frac{\Theta_e(\omega)}{(\omega + i \gamma_m)(\omega + i\gamma_e)} x, y\bigg)
\]
is such that $i \fs_1(\omega)$ is quasi-accretive. Similarly, defining
\[ \Sigma_2 := \{ \omega\in{\mathbb C} \, | \, \re(\omega)\im(\omega+i\gamma_m/2)<0\}, \]
if $\omega \in \Sigma_2$ then $-i \fs_1(\omega)$ is quasi-accretive.
\end{theorem}
\begin{proof}
Let $\fa(\omega)$ be the sesquilinear form
\[ \fa(\omega)(x,y) = \bigg( \frac{\Theta_e(\omega)}{(\omega + i \gamma_m)(\omega + i\gamma_e)} x, y\bigg). \]
Let $\ft(\omega)(x, y) = (\Theta_m(\omega)^{-1} \curl_0 x, \curl_0 y)$, $x,y \in H_0(\curl, \Omega)$. Since $\fa(\omega)$ is bounded for fixed $\omega \in \Sigma_1$, we see that $e^{i \phi}\fa$ is closed and quasi-accretive with domain $\cH$ for all $\phi \in[0, 2\pi)$. Assume that we have already proved that $i \ft(\omega)$ is quasi-accretive for $\omega$ lying in the set $\Sigma_1$. By \cite[Thm VI.1.27, VI.1.31]{MR1335452} the sum $i \fs_1(\omega) = i \fa(\omega) + i \ft(\omega)$ is closed and quasi-accretive on $\dom(\ft(\omega))$.

Hence, it is sufficient to show that if $\omega \in \Sigma_1$ then $i \ft(\omega)$ is quasi-accretive. Equivalently, we must show that $W(\ft(\omega)) \subset {\mathbb H}_- = \{z \in \C : \im z \leq 0 \}$ for all $\omega \in \Sigma_1$. We note that
\[
(\Theta_m(\omega)^{-1} u, u) = \frac{ \bar{\omega}^2 - i \gamma_m \bar{\omega} - (\theta_m^2 u, u)}{| \omega^2 + i \gamma_m \omega - (\theta_m^2 u, u)|^2}
\]
for every $u \in \cH$, $\norma{u} = 1$. Now
\[
\bar{\omega}^2 - i \gamma_m \bar{\omega} - (\theta_m^2 u, u) =  \bigg(\bar{\omega} - \frac{i\gamma_m}{2}\bigg)^2 + \frac{\gamma^2_m}{4} - (\theta_m^2 u, u).
\]
Suppose first that $\omega\in\Sigma_1$ and $\re(\omega)>0$. Then $\omega = -i \gamma_m/2 + r e^{i\phi}$, for some $r > 0$, $\phi \in (0, \pi/2)$. (The case where $\omega \in \Sigma_1$ with $\re(\omega)<0$ follows from the previous case and a reflection argument with respect to the point $-i \gamma_m/2$.) It follows that
\[
\bigg(\bar{\omega} - \frac{i\gamma_m}{2}\bigg)^2 + \frac{\gamma^2_m}{4} - (\theta_m^2 u, u) = r^2 e^{-2i\phi} + \frac{\gamma^2_m}{4} - (\theta_m^2 u, u)
\]
and upon taking the imaginary part, $\im(r^2 e^{-2i\phi} + \frac{\gamma^2_m}{4} - (\theta_m^2 u, u)) = \sin(-2\phi) <0$. Thus,
\begin{multline} \label{eq: final ineq}
\im \ft(\omega)[x] = \im (\Theta_m(\omega)^{-1} \curl_0 x, \curl_0 x) \\
\leq r^2 \sin(-2\phi)  \bigg( \frac{1}{| \omega^2 + i \gamma_m \omega - \theta_m^2|^2} \curl_0 x, \curl_0 x \bigg) \leq 0.
\end{multline}
So $W(\ft(\omega)) \subset {\mathbb{H}}_-$ for all $\omega \in \Sigma_1$ and $i \ft(\omega)$ is accretive for all $\omega \in \Sigma_1$. The final claim follows by noting that when $\omega \in \Sigma_2$, $\im (\ft(\omega)[x]) \geq 0$, by reversing all the inequalities in \eqref{eq: final ineq}.
\end{proof}

\begin{rem}
Theorem \ref{thm: s_1 form_refined} shows that $\fs_1(\omega)$ has a well-defined $m$-accretive representation via the first representation theorem for all $\omega \in \C \setminus (i\R \cup (-i \gamma_m/2 + \R))$. We note \emph{en passant} that the singular sets $\{-i \gamma_e, -i \gamma_m \}$ and $W(\Theta_m)$ are both contained in $\C \setminus \Sigma = (i \R \cup (-i \gamma_m/2 + \R))$.
\end{rem}

\begin{corollary} \label{cor: S_1} Let $\Sigma_1$ and $\Sigma_2$ be as in Theorem \ref{thm: s_1 form_refined}.
For every $\omega \in \Sigma_1$, there exists an operator $\cS_1(\omega)$ such that $i \cS_1(\omega)$ is $m$-accretive and
\[
i(\cS_1(\omega) x, y) = i\fs_1(\omega)[x,y]
\]
for all $x \in \dom(S_1)$, $y \in \dom(\fs_1)$, and $\dom(\cS_1)$ is a core of $\dom(\fs_1)$. Similarly for $\omega\in\Sigma_2$ there
exists an operator $\cS_1(\omega)$ such that $-i\cS_1(\omega)$ is $m$-accretive and
\[
-i(\cS_1(\omega) x, y) = -i\fs_1(\omega)[x,y].
\]
\end{corollary}

\begin{rem}\label{S1rmk}
The operator $\cS_1$ defined in Corollary \ref{cor: S_1} is given, for $\omega\in \Sigma := \Sigma_1 \cup \Sigma_2$ defined in \eqref{def:Sigma}, by
\begin{equation}\label{s1def}
\cS_1(\omega) = \curl ( \Theta_m(\omega)^{-1} \curl_0 \cdot) - \frac{\Theta_e(\omega)}{(i \gamma_m + \omega)(i \gamma_e + \omega)}.
\end{equation}
From (\ref{Thetadef}), we see that $\Theta_m(it)<0$ for all sufficiently large $t\in\R$, and for all $t\geq 0$. Since $\Theta_e(it)<0$
for all $t\in\I\R_+$, $S_1(\omega)$ is self-adjoint and negative for $\omega\in\I\R_+$.
\end{rem}

\begin{prop} \label{prop: symmetry}
The following properties hold.
\begin{enumerate}[label=(\roman*)]
\item $\sigma(\cA) = - \ov{\sigma(\cA)}$.
\item $\sigma(\cL) \setminus \{-i\gamma_e, -i \gamma_m\} = - \ov{\sigma(\cL)} \setminus  \{-i\gamma_e, -i \gamma_m\}$.
\item $\sigma(\cS_1) \cap \Sigma = - \ov{\sigma(\cS_1) \cap \Sigma}$
\end{enumerate}
\end{prop}
\begin{proof}
$(i)$ follows from the equality $\cA = Q \cA_{c} Q^{-1}$ for $Q = \diag(-i,-i,1,1)$, and $\cA_c = - \ov{\cA_c}$ $(ii)$ then follows from $(i)$ due to the equality $\sigma(\cL)\setminus \sigma(-iD) = \sigma(\cA) \setminus \sigma(-iD)$ and the fact that $\sigma(-iD)$ is invariant to the symmetry $\zeta \mapsto - \ov{\zeta}$. $(iii)$ now follows from $(ii)$ in a similar fashion.
\end{proof}

\noindent\textbf{Notation.}  Define multiplication operators in $L^2(\Omega)^3$ by
\begin{equation}\label{Vemdef}  \begin{array}{l}
V_m(\omega) = \frac{\Theta_m(\omega)}{(\omega + i \gamma_m)}, \;\;\;\; \omega\neq -i\gamma_m; \\
   V_e(\omega) = -\frac{\Theta_e(\omega)}{(\omega + i \gamma_e)(\omega + i \gamma_m)}, \;\;\;\; \omega \not\in \{-i\gamma_m,-i\gamma_e\}.
  \end{array}
\end{equation}

\begin{lemma}\label{lemma:bddcls}
Let $\omega \in \Sigma$. Then $\curl_0 \cS_1(\omega)^{-1}$ is closed and bounded in $L^2(\Omega)^3$; $\cS_1(\omega)^{-1} \curl $ and $\curl_0 \cS_1(\omega)^{-1} \curl$ are closable with bounded closure as operators in $L^2(\Omega)^3$.
\end{lemma}
\begin{proof}
Since $\dom(\cS_1(\omega)) \subset H_0(\curl, \Omega) = \dom(\curl_0)$ we have immediately that $\curl_0 \cS_1(\omega)^{-1}$ is closed and bounded, $\cS_1(\omega)^{-1} \curl$ is closable with bounded closure given by $\overline{\cS_1(\omega)^{-1} \curl} = (\curl_0 \cS_1(\omega)^{-*})^*$. It therefore remains to prove that $\curl_0 \cS_1(\omega)^{-1} \curl$ is closable and bounded. We will prove that $\curl_0 \cS_1(\omega)^{-1} \curl$ can be extended as a bounded operator in $L^2(\Omega)^3$, therefore implying that it is closable.\\
Let $B < 0$ be bounded and self-adjoint with the property that $\im(V_e(\omega) + i B) < 0$, where $V_e(\omega)$ is the multiplication operator defined in \eqref{Vemdef}. Then $\im(\cS_1(\omega) + i B) < 0$ for $\omega \in \Sigma_1$ as a consequence of Thm. \ref{thm: s_1 form_refined}. The Lax-Milgram theorem implies immediately that $\cS_1(\omega) + iB$ is boundedly invertible in $L^2(\Omega)^3$. However, a further inspection shows that if $u$ is the weak solution of $(\cS_1(\omega) + iB) u = \curl g$, $g \in L^2(\Omega)^3$, i.e.,
\[
\langle (\Theta_m(\omega))^{-1} \curl_0 u, \curl_0 v \rangle +  \langle (V_e(\omega)+iB) u, v \rangle = \langle g, \curl_0 v \rangle, \quad v \in (C^\infty_c(\Omega))^3,
\]
then for every $\delta < 1$ we have
\[
c_1(\omega) \norma{\curl_0 u}^2 + c_2(\omega) \norma{u}^2 \leq \frac{\norma{g}^2}{4 \delta} + \delta \norma{\curl_0 u}^2,
\]
in which
\[
\begin{split}
c_1(\omega) &= \essinf_{x \in \Omega}|\im \Theta_m(\omega, x)^{-1}| \\
&= |\re \omega| |2 \im \omega + \gamma_m | \, \essinf_{x \in \Omega}\left(\frac{1}{|\omega^2 + i \gamma_m \omega - \theta_m^2(x)|^2} \right),\\[0.3cm]
c_2(\omega) &= \inf_{u \in L^2(\Omega)^3} \frac{| \im (\langle (V_e(\omega)+iB) u, u \rangle) |}{\norma{u}^2} > 0.
\end{split}
\]
From this we deduce that $u \in H_0(\curl, \Omega)$, hence $(\cS_1(\omega) + i B)^{-1}$ maps $\curl L^2(\Omega)^3$ to $H_0(\curl, \Omega)$, or equivalently $\curl_0 (\cS_1(\omega) + i B)^{-1} \curl$ has bounded closure in $L^2(\Omega)^3$.\\
Now let $\omega \in \rho(\cS_1) \cap \Sigma_1$. Then $\cS_1(\omega)^{-1}$ is a bounded operator in $L^2(\Omega)^3$ and we have the resolvent identity
\[
\cS_1(\omega)^{-1} = (\cS_1(\omega) + i B )^{-1} + i \cS_1(\omega)^{-1} B ( \cS_1(\omega) + i B)^{-1}
\]
and then
\[
\cS_1(\omega)^{-1} \curl = (\cS_1(\omega) + i B )^{-1} \curl + i \cS_1(\omega)^{-1} B ( \cS_1(\omega) + i B)^{-1} \curl
\]
is bounded since so is the right-hand side, due to the previous discussion. Now, if $u \in L^2(\Omega)^3$ is the weak solution of
\[
\cS_1(\omega) u = \curl (\Theta_m(\omega))^{-1} \curl_0 u + V_e(\omega) u = \curl g
\]
for some $g \in L^2(\Omega)^3$, $V_e(\omega) u \in L^2(\Omega)^3$ implies $\curl \Theta_m(\omega)^{-1} \curl_0 u \in L^2(\Omega)^3$; then, we may immediately conclude that $u \in \dom(\curl \Theta_m(\omega)^{-1} \curl_0) \subset H_0(\curl, \Omega)$. In particular, $\cS_1(\omega)^{-1} \curl$ is bounded as an operator from $L^2(\Omega)^3$ to $H_0(\curl, \Omega)$, concluding the proof.
\end{proof}

\begin{prop}\label{sigmaLS}
$\sigma(\cL)\cap \Sigma = \sigma(\cS_{1})  \cap \Sigma$ and $\sigma_x(\cL) \cap \Sigma = \sigma_x(\cS_{1})  \cap \Sigma$,
where $x \in \{p, c, r, e\}$, denoting point, continuous, residual, and essential spectrum, respectively.
\end{prop}
\begin{proof}
The proof is along the lines of \cite[Prop. 2.10.1(c)]{TreB}. Without loss of generality, we consider only $\omega \in \Sigma_1$. Note that $\cL(\omega)$ is the sum of a self-adjoint operator and a $\cJ$-self-adjoint bounded operator, where $\cJ = \diag (i, -i) J$, $J$ being the standard complex conjugation; hence, it is easy to check that $\cL(\omega)$ is $\cJ$-self-adjoint for all $\omega \in \Sigma_1$ and $\dom(\cL(\omega))$ does not depend on $\omega \in \Sigma_1$ and it is given by $H_0(\curl, \Omega) \oplus H(\curl, \Omega)$. Due to Theorem \ref{thm: s_1 form_refined}, $i \cS_1(\omega)$, $\omega \in \Sigma_1$ is a well-defined $m$-accretive operator associated with the sesquilinear form $i \fs_1(\omega)$, via the first representation theorem. In particular, $\dom(\cS_1(\omega)) \subset H_0(\curl, \Omega)$ does not depend on $\omega \in \Sigma_1$ and it is a core for $H_0(\curl, \Omega)$. \\
We will first prove that $\sigma(\cS_1) \cap \Sigma_1 \subset \sigma(\cL) \cap \Sigma_1$ and that $\sigma_p(\cS_1) \cap \Sigma_1 = \sigma_p(\cL) \cap \Sigma_1$. If $f \in L^2(\Omega)^3$ and $\omega \in \rho(\cL) \cap \Sigma_1$, then the solution $u$ to the equation $\cL(\omega)(u,v)^t = (f,0)^t$, $v := - i (\omega + i \gamma_m)\Theta_m(\omega)^{-1} \curl_0 u$, is in one-to-one correspondence with the solution of $\cS_1(\omega) u = f$, therefore implying that $\omega \in \rho(\cS_1)$. This also proves that $\sigma_p(\cS_1) \cap \Sigma_1 = \sigma_p(\cL) \cap \Sigma_1$ by arguing in a similar way for $f = 0$. \\
We now prove that  $\sigma(\cS_1) \cap \Sigma_1 \supset \sigma(\cL) \cap \Sigma_1$. Assume that $\omega \in \rho(\cS_1) \cap \Sigma_1$.
Then, at least on $L^2(\Omega)^3 \oplus V_m(\omega)H(\curl, \Omega)$, after recalling \eqref{Vemdef}, we can write the equality
\begin{equation}
\label{eq: res_eq}
\cL(\omega)^{-1} =
\begin{pmatrix}
\cS_1(\omega)^{-1} & - \cS_1(\omega)^{-1} i \curl V_m(\omega)^{-1} \\
V_m(\omega)^{-1} i \curl_0 \cS_1(\omega)^{-1} & V_m(\omega)^{-1} (I + \curl_0 \cS_1(\omega)^{-1} \curl (V_m(\omega))^{-1})
\end{pmatrix}
\end{equation}
so that $\cS_1(\omega) = T(\omega) + V_e(\omega)$, $T(\omega) = \curl \Theta_m(\omega)^{-1} \curl_0$. First note that $L^2(\Omega)^3 \oplus V_m(\omega)H(\curl, \Omega)$ is dense in $L^2(\Omega)^3 \oplus L^2(\Omega)^3$, whenever $\omega \in \Sigma_1$. Hence, it suffices to prove that the right-hand side in \eqref{eq: res_eq} has bounded closure as an operator in $L^2(\Omega)^3 \oplus L^2(\Omega)^3$. This is an immediate consequence of Lemma \ref{lemma:bddcls}, concluding therefore the proof of the inclusion.\\
It remains to prove that $\sigma_e(\cL) \cap \Sigma = \sigma_e(\cS_1) \cap \Sigma$. Due to the previous part of the proof, it is enough to show that $\sigma_p(\cL) \cap \sigma_e(\cL) \cap \Sigma = \sigma_p(\cS_1) \cap \sigma_e(\cS_1)\cap \Sigma$. First, we note that $\cL(\omega)$ is $\cJ$-selfadjoint with respect to $\cJ = \diag(i, -i) J$ and $\cS_1(\omega)$ is $J$-self-adjoint, $J$ being the standard complex conjugation; therefore, \cite[Theorem IX.1.6]{EE} implies $\sigma_{e1}(\cL(\omega)) = \cdots = \sigma_{e4}(\cL(\omega))$ (and similarly for $\cS_1(\omega)$). We will first show that $\sigma_{e2}(\cS_1)) \cap \Sigma \subset \sigma_{e2}(\cL)) \cap \Sigma$. Let $\omega \in \sigma_{e2}(\cS_1)) \cap \Sigma$ and let $u_n$ be a Weyl singular sequence in $\dom(\cS_1)$ such that $\cS_1(\omega) u_n \to 0$. Then, by setting $v_n = - (\omega + i \gamma_m) (\Theta_m(\omega))^{-1} \curl_0 u_n$ and $h_n = (u_n, v_n)/ (\norma{u_n}^2 + \norma{v_n}^2)^{1/2}$ we have $\cL(\omega) h_n = \norma{h_n}^{-1}(\cS_1(\omega) u_n, 0)^t \to 0$.
Let $t \in \R$ be such that $\cS_1(\omega) + i t$ is boundedly invertible (this is possible since for $|t|$ sufficiently big $ \im (\fs_1(\omega) + it)[u] < 0$, hence by the Lax-Milgram theorem we conclude that $\cS_1(\omega) + i t$ is boundedly invertible).  Now from $\cS_1(\omega) u_n \to 0$ we deduce
\[
v_n = - (\omega + i \gamma_m) (\Theta_m(\omega))^{-1} \curl_0 (\cS_1(\omega) + it)^{-1}(\cS_1(\omega) + it) u_n \rightharpoonup 0
\]
since $\curl_0 (\cS_1(\omega) + it)^{-1}$ is a bounded operator, $\cS_1(\omega) u_n \to 0$ and $u_n \rightharpoonup 0$. Hence $h_n$ is a Weyl sequence for $\cL(\omega)$, and $\omega \in \sigma_{e2}(\cL)$.
\\
We will now show that $\sigma_{e4}(\cS_1) \cap \Sigma \supset \sigma_{e4}(\cL) \cap \Sigma$. We can assume without loss of generality that we are inside $\Sigma_1$. Now assume that $\omega \in \Sigma_1$ but $\omega \notin \sigma_{e4}(\cS_1)$. Then there exists a compact operator such that $0 \in \rho(\cS_1(\omega) + K)$. According to \eqref{eq: res_eq}, if $\cK = \diag(K,0)$, then
{\small
\begin{multline*}
(\cL(\omega) + \cK)^{-1} \\
\quad \quad = \begin{pmatrix}
(\cS_1(\omega) + K)^{-1} & - (\cS_1(\omega) + K)^{-1} i \curl V_e(\omega)^{-1} \\
V_e(\omega)^{-1} i \curl_0 (\cS_1(\omega) + K)^{-1} & V_e(\omega)^{-1} (I + \curl_0 (\cS_1(\omega) + K)^{-1} \curl (V_e(\omega))^{-1})
\end{pmatrix}
\end{multline*}
}
In order to conclude, we then just need to show that $\curl_0 (\cS_1(\omega) + K)^{-1} \curl$ is bounded. But this can be proved as in the proof of Lemma \ref{lemma:bddcls} by first showing that $\curl_0 (\cS_1(\omega) + K + iB)^{-1} \curl$ is bounded for some bounded operator $B < 0$, and then by using the resolvent identity
{\small
\[
(\cS_1(\omega) + K)^{-1} \curl = (\cS_1(\omega) + K + i B )^{-1} \curl + i(\cS_1(\omega) + K)^{-1} B ( \cS_1(\omega) + K + i B)^{-1} \curl.
\]
}
Altogether, $(\cL(\omega) + \cK)^{-1}$ is bounded as an operator in $L^2(\Omega)^3 \oplus L^2(\Omega)^3$, hence $\omega \notin \sigma_{e4}(\cL) \cap \Sigma_1$.
\end{proof}

According to Proposition \ref{sigmaLS}, $\sigma(\cL)\cap \Sigma$ can be enclosed in $W(\cS_1) \cap \Sigma$.

\begin{theorem} \label{thm:refnumran}
$\sigma(\cL) \cap \Sigma \subset W(\cS_1) \cap \Sigma$. Moreover, the following explicit enclosure holds:
\[
\sigma(\cL) \setminus \overline{W(\Theta_m)}\subset \Gamma \setminus \overline{W(\Theta_m)}
\]
\begin{multline*}
\Gamma := \{\omega \in \C : \re \omega = 0, \im \omega \in (- \gamma_e, 0) \setminus \{\gamma_m\} \} \\
\cup \{\omega \in \C : \re \omega \neq 0, \im \omega \geq - (\gamma_e + \gamma_m)/2, \textup{\eqref{enc} holds}\}
\end{multline*}
\end{theorem}
\begin{proof}
Note that due to Remark \ref{S1rmk}, the operator $\cS_1(it)$, $t \in \R \setminus (-\gamma_m, 0)$ is symmetric and since for fixed $t$ it is the sum of a semibounded self-adjoint operator and a bounded self-adjoint operator, $\cS_1(it)$ is self-adjoint and semibounded. The proof of Proposition \ref{sigmaLS} therefore extends to $\omega \in i\R \setminus (-i \gamma_m, 0)$, giving $\sigma(\cL) \setminus \overline{W(\Theta_m)} = \sigma(\cS_1) \setminus \overline{W(\Theta_m)}$ (note that $W(\Theta_m) \subset (- i \gamma_m, 0) \cup (-i \gamma_m/2 - s, -i \gamma_m/2 + s)$ for some $s > 0$ depending on $\theta_m$).\\
Let then $\omega \in \sigma_{\rm app}(\cS_1) \setminus \overline{W(\Theta_m)}$. There exists $u_n \in \dom(\cS_1)$, $\norma{u_n} = 1$, $n \in \N$, such that $\cS_1(\omega) u_n \to 0$. In particular,
\begin{equation}
\label{NRS1}
\begin{cases}
&\im \langle \Theta_m(\omega)^{-1} \curl_0 u_n, \curl_0 u_n \rangle = \im \left\langle \frac{\Theta_e(\omega)}{(\omega + i\gamma_e)(\omega + i \gamma_m)} u_n, u_n \right\rangle + \eps_n \\
&\re \langle \Theta_m(\omega)^{-1} \curl_0 u_n, \curl_0 u_n \rangle = \re \left\langle \frac{\Theta_e(\omega)}{(\omega + i\gamma_e)(\omega + i \gamma_m)} u_n, u_n \right\rangle + \eps_n
\end{cases}
\end{equation}
with $\eps_n \to 0$ as $n \to \infty$. Let $\omega = x + i y$. Then \eqref{NRS1} can be written explicitly as
\begin{equation}
\label{NRS1_2}
\begin{cases}
&\begin{aligned}
&\left\langle \frac{x (2 y + \gamma_m)}{|\Theta_m(\omega)|^2} \curl_0 u_n, \curl_0 u_n \right\rangle = \frac{x \gamma_m}{x^2 + (y + \gamma_m)^2} \\
&\hspace{2cm}- \frac{x(2y + \gamma_e + \gamma_m)}{(x^2 + (y + \gamma_m)^2)(x^2 + (y + \gamma_e)^2)} \langle \theta_e^2 u_n, u_n \rangle + \eps_n
\end{aligned}\\[1cm]
&\begin{aligned}
&\left\langle \frac{x^2 - y^2 - \gamma_m y - \theta_m^2}{|\Theta_m(\omega)|^2} \curl_0 u_n, \curl_0 u_n \right\rangle = \frac{x^2 + y^2}{x^2 + (y + \gamma_m)^2} \\
&\hspace{2cm}- \frac{(x^2 - y^2 - y(\gamma_e + \gamma_m) - \gamma_e \gamma_m)}{(x^2 + (y + \gamma_m)^2)(x^2 + (y + \gamma_e)^2)}\langle \theta_e^2 u_n, u_n \rangle + \eps_n
\end{aligned}
\end{cases}
\end{equation}
If $x = 0$, $y < - \gamma_e$ then the second equation reads
\begin{multline*}
\left\langle \frac{- y^2 - \gamma_m y - \theta_m^2}{|\Theta_m(\omega)|^2} \curl_0 u_n, \curl_0 u_n \right\rangle \\
= \frac{y^2}{(y + \gamma_m)^2} - \frac{(- (y + \gamma_e)(y + \gamma_m))}{((y + \gamma_m)^2)(y + \gamma_e)^2)}\langle \theta_e^2 u_n, u_n \rangle + \eps_n
\end{multline*}
and since $y < - \gamma_e < - \gamma_m$, the left-hand side is negative while the right-hand side is strictly positive for sufficiently big $n$. Similarly, if $y > 0$, the left-hand side is negative while the right-hand side is strictly positive for sufficiently big $n$, a contradiction. Therefore, if $x = 0$, $y \in (- \gamma_e, 0)$. \\
Now, assume $x \neq 0$. We can then divide by $x$ in the first equation of \eqref{NRS1_2}. Then we see immediately that if $y \leq - (\gamma_e + \gamma_m)/2 < - \gamma_m/2$ then the left-hand side is negative while the right-hand side is strictly positive, a contradiction. Therefore, for $x \neq 0$, $y > - (\gamma_e + \gamma_m)/2$.
\end{proof}

\section{Decomposition of the essential spectrum}
\label{sec:ess_spec}

In this section we will adapt the strategy of proof recently used for the analogous decomposition of the essential spectrum for the time-harmonic Maxwell system with non-trivial conductivity in the recent article \cite{BFMT}. For the convenience of the reader we will state and prove all the required results.

Without loss of generality we assume that $\Omega$ is unbounded, the bounded case being substantially simpler, see Remark \ref{rem:bdd_dec} below. Let $\Omega_R = \Omega \cap B(0,R)$, for $R>0$. For any $\delta>0$ we assume that
the functions $\theta_e$ and $\theta_m$ admit a decomposition
\begin{equation} \label{eq:coeffs-infty}
\theta_e(x) = \theta^c_{e}(x) + \theta^\delta_{e}(x) + \theta_e^0, \quad \theta_m(x) = \theta^c_{m}(x) + \theta^\delta_{m}(x) + \theta_m^0
\end{equation}
for all $x \in \Omega$, where $\theta^c_{e}, \theta^c_{m}, $ have compact support in $\Omega_R$ (for some sufficiently large $R$ depending on $\delta$), $\theta^\delta_{e}$, $\theta^\delta_{m}$, are bounded multiplication operators with norm less than $\delta$, and $\theta_e^0, \theta_m^0$ are real constants, representing the asymptotic values of $\theta_e$ and $\theta_m$. In particular,
\[
\lim_{R \to \infty} \sup_{|x| > R} |\theta_*(x) - \theta_*^0| = 0, \quad * = e,m.
\]
Corresponding to this decomposition of $\theta_e$ and $\theta_m$ we also introduce `limits at infinity' of the functions $\Theta_e$ and
$\Theta_m$ in \eqref{Thetadef}, namely
\begin{equation}\label{Thetainfdef}
\Theta_{e,\infty}(\omega) = \omega^2 + \omega i\gamma_e -(\theta_e^0)^2, \;\;\;
\Theta_{m,\infty}(\omega) = \omega^2 + \omega i\gamma_m -(\theta_m^0)^2,
\end{equation}
and of the functions $V_e$ and $V_m$ appearing in \eqref{Vemdef}, namely
\begin{equation}
\label{Veminfty}
V_{m,\infty}(\omega) = \frac{\Theta_{m,\infty}(\omega)}{(\omega+i\gamma_m)}, \;\;\;
 V_{e,\infty}(\omega) = \frac{\Theta_{e,\infty}(\omega)}{(\omega+i\gamma_e)(\omega+i\gamma_m)}.
\end{equation}

We use the classical Helmholtz decomposition $L^2(\Omega)^3\!=\!\nabla \dot H^1_0(\Omega) \oplus H(\Div 0,\Omega)$, see e.g.\ \cite[Lemma 11]{MR3942228}, and we denote by $P_{\ker(\Div)}$ the associated orthogonal projection onto $H(\Div 0,\Omega)$. The following result is stated in \cite[Proposition 5.1]{BFMT} and in a less general setting in \cite[Lemma 23]{MR3942228}.

\begin{prop}
\label{thm: compactness}
Let $m:\Omega\to \C^{3\times 3}$ be a locally bounded function such that
\begin{equation}
\label{eq:limit-gen}
\lim_{R\to\infty}\sup_{\|x\|>R} \|m(x)\|=0.
\end{equation}
Then $mP_{\ker(\Div)}$ is compact from $(H(\curl, \Omega), \norma{\cdot}_{H(\curl, \Omega)})$ to  $( L^2(\Omega)^3, \norma{\cdot}_{ L^2(\Omega)^3})$.
\end{prop}

\begin{proof}
Given $\delta>0$, there exist a bounded operator $m_\delta$ (which we identify with the corresponding multiplication operator in $L^2(\Omega)^3$) with $\|m_\delta\|<\delta$, and a function $m^\delta_c$ which is compactly supported in $\Omega_R := \Omega\cap B(0,R)$ for large $R>0$, such that $m = m^\delta_c + m_\delta$. We claim that $m^\delta_cP_{\ker(\Div)}$ is compact for every $\delta > 0$. Note that $\norma{m P_{\ker(\Div)} \!-\! m^\delta_c P_{\ker(\Div)}}_{\cB(H(\curl,\Omega),  L^2(\Omega)^3)} \leq  \delta$ vanishes as $\delta \to 0$; therefore $mP_{\ker(\Div)}$ is compact as limit of the compact operators $m^\delta_c P_{\ker(\Div)}$.

Define $\chi_R$ to be a $C^\infty$ cut-off function, $\chi_R = 1$ on $ \supp(m_c)\subset \Omega_R$ and $\chi = 0$ in $\R^3 \setminus \overline{\Omega_R}$. There exists $C_R>0$ such that, for $u\in H(\curl,\Omega)$,
$$
\|(\chi_R P_{\ker(\Div)} u)|_{\Omega_R}\|_{H(\curl, \Omega_R) \cap H(\Div, \Omega_R)}
\leq C_R\|u\|_{H(\curl, \Omega)},
$$
where we have used that $\Div (\chi_R P_{\ker(\Div)}u)=\nabla \chi_R\cdot P_{\ker(\Div)}u$ and the identity
$\curl (\chi_R P_{\ker(\Div)}u)=\nabla \chi_R\times P_{\ker(\Div)}u+\chi_R \curl u$, which holds since $\curl P_{\ker(\Div)}u=\curl u$.
Finally, $m_c P_{\ker(\Div)}$ is seen to be compact by rewriting it as follows
$$m_c P_{\ker(\Div)} u = m_c \iota(\chi_R P_{\ker(\Div)} u)|_{\Omega_R};$$
$\iota$ is the compact embedding of ${H_0(\curl, \Omega_R)} \cap H(\Div, \Omega_R)$ in $L^2(\Omega_R)^3\!$, see~\cite{MR561375}.
\end{proof}

\begin{rem} \label{rem:compactness_bdd}
If $\Omega$ is bounded, the claim of Prop \ref{thm: compactness} holds true without assuming \eqref{eq:limit-gen}. This is a direct consequence of the compact embedding of $H_0(\curl, \Omega) \cap H(\Div0, \Omega)$ into $L^2(\Omega)^3$, whenever $\Omega$ is bounded.
\end{rem}

\begin{definition}
\label{Smdef}
For $\omega\in \Sigma = \Sigma_1 \cup \Sigma_2$,
we define rational pencils of closed operators acting in the Hilbert space 
$H(\Div 0,\Omega)$ equipped with the $ L^2(\Omega)^3$-norm by
\begin{align*}
&\begin{array}{l}
\cC_m(\omega) \!:=\! \curl (\Theta_m(\omega))^{-1} \curl_{0}, \;\;\;\;\;\;
\cS_m(\omega) \!:=\!  \cC_m(\omega) - V_{e,\infty}(\omega),  \\[2mm]
\dom(\cC_m(\omega)) = \dom(\cS_m(\omega)) \\
 \hspace{2cm} \!:=\! \{u \in H_0(\curl,\Omega){\cap H(\Div 0,\Omega)} \; :\;
(\Theta_m(\omega))^{-1}\curl u \!\in\! H(\curl,\Omega) \},
\end{array}
\intertext{
and}
&\begin{array}{l}
\!\cC_{\infty}(\omega) \!:=\!  (\Theta_{m,\infty}(\omega))^{-1} \curl \curl_{0},\;\;\;\;\;\;
\!\cS_{\infty}(\omega) \!:=\!  \cC_{\infty}(\omega) -  V_{e,\infty}(\omega),  \\[2mm]
\!\dom(\cC_{\infty}(\omega)) = \!\dom(\cS_{\infty}(\omega)) \!:=\! \{u \in H_0(\curl,\Omega) {\cap H(\Div 0,\Omega)} :
\curl u \!\in\! H(\curl,\Omega) \}
\end{array}
\end{align*}
where $V_{e,\infty}(\omega)$ is defined in \eqref{Veminfty}.
\end{definition}

\textbf{Notation.} For $\omega \in \C \setminus \{-i\gamma_e, -i \gamma_m\}$ define the function
\begin{equation} \label{def:f}
f(\omega) = \frac{\Theta_{m,\infty}(\omega) \Theta_{e,\infty}(\omega)}{(\omega + i \gamma_e)(\omega + i \gamma_m)}
\end{equation}

If $\theta_m$ is not differentiable, it may happen that $\dom(\cS_m)$ is $\omega$-dependent and even that
$\dom(\cS_m(\omega)) \cap \dom(\cS_{\infty}(\omega)) = \{ 0 \}$ for suitably chosen $\omega \in \Sigma$. In spite of this, the following result holds (the version for
the non-self-adjoint time-harmonic Maxwell system was proved in \cite[Proposition 5.4]{BFMT}).

\begin{prop}
\label{thm: difference res}
If $\theta_e$, $\theta_m$ satisfy \eqref{eq:coeffs-infty} and $\cS_m$, $\cS_{\infty}$ are as in Definition {\rm \ref{Smdef}},  then $\sigma_{ek}(\cS_m) \!=\! \sigma_{ek}(\cS_{\infty})$ for $k=1,2,3,4$, and hence	
\begin{align*}
   \sigma_{ek}(\cS_m) \cap \Sigma \!=\! \bigg\{ \omega \in \Sigma :\, f(\omega) = t, \:\, t \in \sigma_{ek}(\curl \curl_0|_{H(\Div 0, \Omega)}) \bigg\}.
\end{align*}
where $f$ is the function defined in \eqref{def:f}.
\end{prop}
\begin{proof}
We follow the proof of \cite[Proposition 5.4]{BFMT}. Let $\omega\!\in\!\Sigma$ and to shorten the notation set $z_\omega\!:= V_{e,\infty}(\omega)\!$. Then $\omega\!\in\!\sigma_{ek}(\cS_m)$ if and only if $0\!\in\! \sigma_{ek}(\cC_m(\omega) - z_\omega)$ and $\omega\!\in\!\sigma_{ek}(\cS_\infty)$ if and only if $ 0 \!\in\! \sigma_{ek}(\cC_\infty(\omega) - z_\omega)$ where, $\cC_m, \cC_\infty$ are the operator functions defined in Definition \ref{Smdef}, for $\omega \in \Sigma$.\\
Since the quadratic form ${\mathfrak c}_m(\cdot)$ associated with $\cC_m(\cdot)$ and the form ${\mathfrak c}_\infty$ associated with $\cC_\infty(\cdot)$ have the same domain
$\dom {\mathfrak c}_m(\cdot) \!=\! \dom {\mathfrak c}_\infty \!=\! H_0(\curl,\Omega)$, the second resolvent identity takes the~form
\begin{multline}
\label{eq:form-2nd-res.-id}
   (\cC_m(\omega)\!-\!z_\omega)^{-1} \!- (\cC_\infty(\omega)\!-\!z_\omega)^{-1} \\
   =\! \big( \curl_0 (\cC_m(\omega)^*\!-\!\overline{z_\omega})^{-1} \big)^{\!*} (\Theta_{m,\infty}(\omega)^{-1} \!\!-\Theta_m(\omega)^{-1}) \curl_0 (\cC_\infty(\omega)\!-\!z_\omega)^{-1}
\end{multline}
for $\omega\!\in\! \Sigma \cap (\rho(\cS_m) \cap \rho(\cS_\infty))$. In fact, for arbitrary $u$, $v \in  L^2(\Omega)^3$ and $\omega\!\in\! \Sigma \cap (\rho(\cS_m) \cap \rho(\cS_\infty))$, we can write
\begin{align*}
  &\big\langle \big( (\cC_m(\omega)-z_\omega)^{-1} - (\cC_\infty(\omega)-z_\omega)^{-1} \big) u,v \big\rangle\\
&= \big\langle  (\cC_m(\omega)-z_\omega)^{-1}  u,v \big\rangle - \big\langle u, (\cC_\infty(\omega)-z_\omega)^{-*} v \big\rangle\\
&= \big\langle  (\cC_m(\omega)-z_\omega)^{-1} u,(\cC_\infty(\omega)^*-\overline{z_\omega})(\cC_\infty(\omega)^*-\overline{z_\omega})^{-1}v \big\rangle \\
&\hspace{2cm}- \big\langle  (\cC_m(\omega)-z_\omega)  (\cC_m(\omega)-z_\omega)^{-1}  u, (\cC_\infty(\omega)^*-\overline{z_\omega})^{-1} v \big\rangle\\
	&= ( {\mathfrak c}_\infty(\omega) - {\mathfrak c}_m(\omega) ) \big[	(\cC_m(\omega)-z_\omega)^{-1}  u, (\cC_\infty(\omega)^*-\overline{z_\omega})^{-1} v  \big];
\end{align*}
together with ${\mathfrak c}_m(\omega) = \langle\Theta_m(\omega)^{-1} \curl_0 \cdot, \curl_0 \cdot\rangle$ and analogously for ${\mathfrak c}_\infty(\omega)$, the identity \eqref{eq:form-2nd-res.-id}~follows.
The first factor on the right-hand side of \eqref{eq:form-2nd-res.-id} is bounded since $\dom \cC_m(\cdot) \subset \dom \curl_0$.
By assumption \eqref{eq:coeffs-infty}, for fixed $\omega \in \Sigma$, \textcolor{black}{condition \eqref{eq:limit-gen} of Proposition \ref{thm: compactness}
is satisfied by} $(\Theta_m(\omega)^{-1}-\Theta_\infty(\omega)^{-1})$ and thus the operator $(\Theta_m(\omega)^{-1}-\Theta_\infty(\omega)^{-1}) P_{\ker\Div}$ is com\-pact from $H(\curl,\Omega)$ to $H(\Div 0,\Omega)\subset L^2(\Omega)^3$. The boundedness of $\curl_0 (\cC_\infty(\omega)-z_\omega)^{-1}$ from $H(\Div 0, \Omega)$ to $H(\curl,\Omega)$ follows from
\[
\begin{split}
\curl\curl_0 (\cC_\infty(\omega)-z_\omega)^{-1} &= \curl \curl_0 (\curl \curl_0 - f(\omega))^{-1} \Theta_m(\omega)\\
&= I + f(\omega)(\curl \curl_0 - f(\omega))^{-1}
\end{split}
\]
where $f$ is defined in \eqref{def:f}. Now, $f(it) < 0$, $t \in \R$, hence
$$0 \leq I + f(it)(\curl \curl_0 - f(it))^{-1} \leq I,$$
and the boundedness for $\omega = it$ follows. For a general $\omega \in (\rho(S_{\infty})\cap \Sigma)$, we have
\[
(\cC_\infty(\omega)-z_\omega)^{-1} = (\cC_\infty(it)-z_{it})^{-1} + (\cC_\infty(it)-z_{it})^{-1} (z_{it} - z_{\omega})(\cC_\infty(\omega)-z_\omega)^{-1}
\]
hence, upon applying $\curl \curl_0$, $\curl \curl_0 (\cC_\infty(\omega)-z_\omega)^{-1}$ is seen to be bounded.\\
Altogether, the operator
\begin{align*}
 &(\Theta_{m,\infty}(\omega)^{-1} -\Theta_m(\omega)^{-1}) \curl_0 (\cC_\infty-z_\omega)^{-1}  \\
 &\hspace{3cm}= (\Theta_{m,\infty}(\omega)^{-1} -\Theta_m(\omega)^{-1}) P_{\ker(\Div)}  \curl_0 (\cC_\infty-z_\omega)^{-1}
\end{align*}
is compact.
Hence, by \eqref{eq:form-2nd-res.-id}, the resolvent difference of $\cS_m(\omega)$ and $\cS_\infty(\omega)$ is compact and, by \cite[Thm.\ IX.2.4]{EE},
$\sigma_{ek}(\cS_m(\omega))=\sigma_{ek}(\cS_\infty(\omega))$ follows for all $k=1,2,3,4$, $\omega \in \Sigma$, \textcolor{black}{hence $0
\in \sigma_{ek}(\cS_m(\omega))$ if and only if $0 \in \sigma_{ek}(\cS_\infty(\omega))$, for $\omega\in\Sigma$. This means that} $\sigma_{ek}(\cS_m) \cap \Sigma=\sigma_{ek}(\cS_\infty) \cap \Sigma$.
\end{proof}

\begin{rem} \label{rem: unif_bound_Cinfty}
In the proof of Prop. \ref{thm: difference res}, it is shown that for $\omega = it$, $t > 0$,  $\norma{\curl_0 (\cC_\infty(\omega)\!-\!z_\omega)^{-1}}_{\cB(H(\Div 0, \Omega),H(\curl,\Omega))} \leq C$, where the constant $C>0$ does not depend on $\Omega$. This will be important in Section \ref{sec:lim_ess_spec}, where families of domains are considered.
\end{rem}

We further state the following abstract result regarding the spectrum of triangular block operator matrices, a proof of which can be found in \cite[Theorem 8.1]{BFMT}. Following \cite[Chp.IX, p.414]{EE}, given a linear operator $T$ densely defined in $\cH$, we set $\sigma^*_{e2}(T) = \{ \omega \in \C : {\rm def} (T - \omega) = \infty \}$, with the convention that ${\rm def} (T - \omega) = \infty$ if $\ran (T - \omega)$ is not closed.

\begin{theorem}
\label{thm: ess spec}
Let $\cA$ be defined by
\[
\cA = \begin{pmatrix}
A & 0 \\
C & D
\end{pmatrix}
\]
with $A$, $D$ are densely defined, $C$, $D$ are closable, $\dom(A) \subset \dom(C)$ and $\rho(A) \neq \emptyset$. Then
\begin{equation}
\label{se2}
  \big( \sigma_{e2}(A) \setminus \sigma_{e2}^*(\overline{D} ) \big) \cup \sigma_{e2}(\overline{D} ) \subset \sigma_{e2}(\overline{\cA})
	\subset \sigma_{e2}(A)  \cup \sigma_{e2}(\overline{D} ),
\vspace{-1mm}	
\end{equation}
and hence
\[
  \sigma_{e2}(\overline{\cA}) \cup \big( \sigma_{e2}(A) \cap \sigma_{e2}^*(\overline{D} ) \big)
	= \sigma_{e2}(A)  \cup \sigma_{e2}(\overline{D} );
\]
in particular, if  $\sigma_{e2}^*(\overline{D})  = \sigma_{e2}(\overline{D})$ or if $ \sigma_{e2}(A) \cap \sigma_{e2}^*(\overline{D})=\emptyset$, then
\[
  \sigma_{e2}(\overline{\cA}) = \sigma_{e2}(A)  \cup \sigma_{e2}(\overline{D} ).
\]
\end{theorem}

We are now in position to prove the following theorem, which yields a decomposition of the $\sigma_{e}(\cL)$ as the union of the essential spectrum of the constant-coefficient pencil $\cS_{\infty}$ and the essential spectrum of the pencil of bounded multiplication operators $V_e(\cdot)$, compressed to gradient fields. For the convenience of the reader, the relations between the several different operators and their essential spectra are represented in Fig. \ref{fig:graph_ALS}.

\begin{theorem}
\label{sigma-ess}
Suppose that $\theta_e$, $\theta_m$ satisfy the limiting assumption \eqref{eq:coeffs-infty}. Let $P_\nabla := \id - P_{\ker(\Div)}$ be the orthogonal projection from $ L^2(\Omega)^3\!=\!\nabla \dot H^1_0(\Omega) \oplus H(\Div 0,\Omega)$ onto $\nabla \dot{H}^1_0(\Omega)$. Let
$G(\omega)$ denote the operator
\begin{equation}\label{B1def}
\mbox{$G(\omega) = -P_{\nabla} V_e(\omega) P_{\nabla}$, with $\dom (G(\omega))=\nabla \dot{H}^1_0(\Omega)$},
\end{equation}
viewed as an operator from the space $\nabla \dot{H}^1_0(\Omega)$ to $\nabla \dot{H}^1_0(\Omega)$. Then
\[
   \sigma_{ek}(\cS_1) \cap \Sigma = (\sigma_{ek}(\cS_\infty) \cup \sigma_{ek}(G)) \cap \Sigma, \quad k=1,2,3,4,
\]
where $\sigma_{ek}(\cS_\infty)$ is described in Prop. \ref{prop: spectrum infty}, and
\[
\sigma_{ek}(G) \subset
\begin{cases}- i [0, \gamma_e), \quad &\textup{if $- \frac{\gamma_e^2}{4} + \norma{\theta^2_e}_{\infty} \leq 0$,} \\
- i [0, \gamma_e) \cup \left(-d_e -i \frac{\gamma_e}{2}, d_e -i \frac{\gamma_e}{2}\right) \quad &\textup{if $- \frac{\gamma_e^2}{4} + \norma{\theta^2_e}_{\infty} > 0$.}
\end{cases}
\]
with $d_e \in \left(- \sqrt{- \frac{\gamma_e^2}{4} + \norma{\theta^2_e}_{\infty}}, \sqrt{- \frac{\gamma_e^2}{4} + \norma{\theta^2_e}_{\infty}} \right)$.
\end{theorem}
\begin{proof}
Let $\omega\in\Sigma$. The operator $M(\omega):=(V_e(\omega)-V_{e,\infty}(\omega))P_{\ker(\Div)}$ in $ L^2(\Omega)^3$ is $\curl_0$-compact by Proposition~\ref{thm: compactness}, and hence $\cC_m(\omega)$-compact, since $\dom(\cC_m) \subset \dom(\curl_0)$, where $\cC_m(\omega) = \curl \Theta_m(\omega)^{-1} \curl_0$ is defined in Definition \ref{Smdef}. Since $\cS_1(\omega)=\cC_m(\omega) + V_e(\omega)$, with $V_e(\omega)$ bounded multiplication operator, bounded sequences in the $\cS_1(\omega)$-graph norm have bounded $\cC_m(\omega)$-graph norms. Hence $M(\omega)$ is $\cS_1(\omega)$-compact which
yields $\sigma_e(\cS_1(\omega))=\sigma_e(\cS_1(\omega)+M(\omega))$.

Since $\nabla \dot{H}^1_0(\Omega)\subset\ker(\curl_0)$ and hence $\curl_0 P_\nabla = P_\nabla \curl=0$,
$\nabla \dot{H}^1_0(\Omega)$ is a reducing subspace for $\curl \Theta_m(\omega)^{-1} \curl_0$.
 Therefore the operator
\begin{equation}\label{eq:omcT}
\begin{aligned}
&\cT(\omega):= \cS_1(\omega)+M(\omega) \\
&= \curl \Theta_m(\omega)^{-1} \curl_0 - V_e(\omega)P_\nabla - V_e(\omega) P_{\ker(\Div)} + (V_e(\omega)-V_{e,\infty}(\omega))P_{\ker(\Div)} \\
&= \cC_m(\omega)-V_e(\omega) P_\nabla - V_{e,\infty}(\omega) P_{\ker(\Div)}
\end{aligned}	
\end{equation}
which is a bounded perturbation of $\cC_m(\omega)$ admits an operator matrix representation with respect to the decomposition
$ L^2(\Omega)^3= \nabla \dot{H}^1_0(\Omega) \oplus H(\Div0, \Omega)$ given \vspace{-1mm} by
\begin{align}
\cT(\omega) &=
\begin{pmatrix}
\hspace{5.7mm} P_\nabla \cT(\omega) |_{\nabla \dot H^1_0(\Omega)} &  \hspace{6.5mm} P_\nabla \cT(\omega) |_{H(\Div0, \Omega)}\\
P_{\ker\Div} \cT(\omega) |_{\nabla \dot H^1_0(\Omega)} & P_{\ker\Div} \cT(\omega) |_{H(\Div0, \Omega)} &
\end{pmatrix}
\nonumber
\\
&=
\begin{pmatrix}
\hspace{5.5mm} -P_\nabla V_e(\omega) |_{\nabla \dot H^1_0(\Omega)}  & 0 \\
-P_{\ker\Div} V_e(\omega) |_{\nabla \dot H^1_0(\Omega)} & P_{\ker\Div} (\cC_m(\omega)-V_{e,\infty}(\omega)) |_{H(\Div0, \Omega)}\\
\end{pmatrix}
\nonumber
\\
&=
\begin{pmatrix}
\hspace{5.7mm} G(\omega) &  0  \\
P_{\ker\Div} V_e(\omega) |_{\nabla \dot H^1_0(\Omega)} &  \cS_m(\omega)
\end{pmatrix}.
\label{eq: A^0}
\\[-7mm] \nonumber
\end{align}
with domain $\dom(\cT(\omega))=\nabla \dot{H}^1_0(\Omega)\oplus \dom(\cS_m(\omega))$. Apart from $\cS_m(\omega)$, the other two matrix entries in $\cT(\omega)$ are bounded and everywhere defined,
and $\sigma_{e2}(\cS_m(\omega)) = \sigma^*_{e2}(\cS_m(\omega))$, due to $J$-self-adjointness.
Thus Theorem \ref{thm: ess spec} and Proposition~\ref{thm: difference res} yield
\vspace{-1mm} that
$$
    \sigma_{e2}(\cT(\omega))
		= \sigma_{e2}(\cS_m(\omega)) \cup \sigma_{e2}(G(\omega))
		= \sigma_{e2}(\cS_\infty(\omega)) \cup \sigma_{e2}(G(\omega))
$$
and hence, since $\omega\in\Sigma$ was arbitrary,
\begin{align*}
\sigma_{e2}(\cS_1) \cap \Sigma&=\sigma_{e2}(\cS_1+M) \cap \Sigma=\sigma_{e2}(\cT) \cap \Sigma = (\sigma_{e2}(\cS_\infty) \cup \sigma_{e2}(G)) \cap \Sigma.
\qedhere
\\[-6mm]
\end{align*}
\end{proof}

\begin{rem} \label{rem:bdd_dec}
If $\Omega$ is bounded, we claim that $\sigma_{ek}(\cS_1) = \sigma_{ek}(G)$ for $k = 1,2,3,4$. To see this, one applies the Helmholtz decomposition in $\nabla H^1_0(\Omega) \oplus H(\Div 0, \Omega)$ to obtain the block operator matrix representation
\[
\cT(\omega) = \begin{pmatrix}
P_{\nabla}\cS_1(\omega)P_{\nabla} & P_{\nabla}\cS_1(\omega)P_{\ker \Div} \\
P_{\ker \Div}\cS_1(\omega)P_{\nabla} & P_{\ker \Div}\cS_1(\omega)P_{\ker \Div}
\end{pmatrix}
\]
of $\cS_1(\omega)$. In particular, $\sigma_{ek}(\cS_1) = \sigma_{ek}(\cT)$. Since $\Omega$ is bounded, $P_{\ker \Div}\cS_1(\omega)P_{\ker \Div}$ has compact resolvent. The term $P_{\nabla}\cS_1(\omega)P_{\ker \Div}$ is compact as a consequence of Remark \ref{rem:compactness_bdd}, and therefore can be discarded to leave a lower triangular operator matrix. The claim then follows arguing as in the proof of Theorem \ref{sigma-ess}, by recalling that $P_\nabla \cS_1(\omega) P_\nabla = G(\omega)$.
%
%
\end{rem}

\begin{figure} \label{fig:graph_ALS}
\centering
\includegraphics{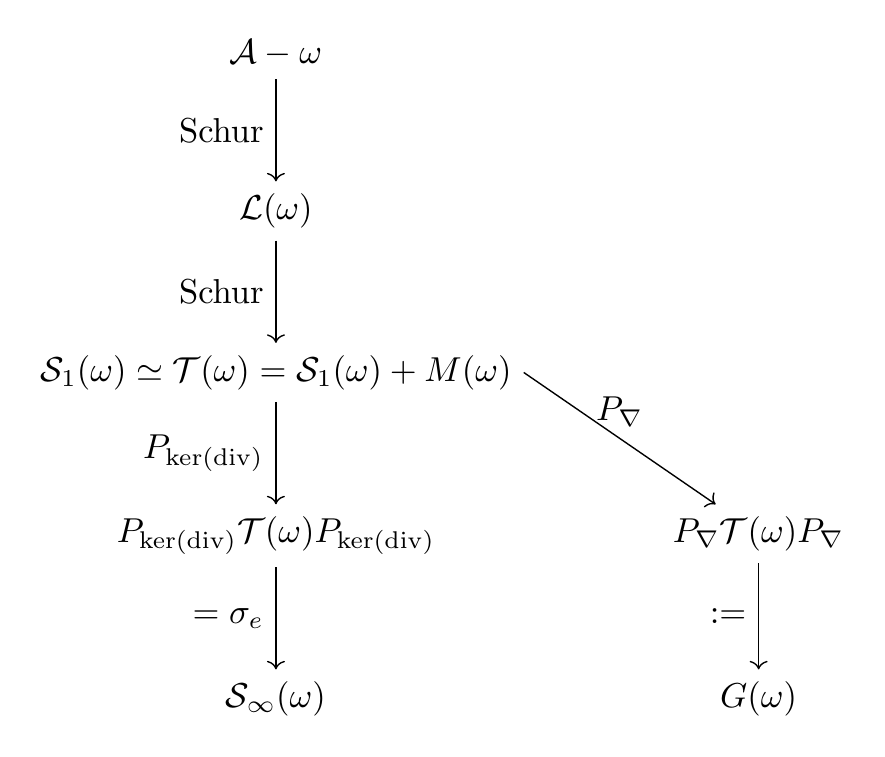}
\caption{Relations between $\cA$, $\cL$ and $\cS_1$. $\sigma_e(\cS_1)$ then decomposes in $\sigma_e(\cS_\infty) \cup \sigma_e(G)$, according to Thm \ref{sigma-ess}}
\end{figure}

\section{Spectrum of the reduced operators}
\label{sec:red_op}

Recall that
\[\cS_{1,\infty}(\omega)|_{\ker(\Div)} = \cS_{\infty}(\omega)\]
for all $\omega \in \Sigma$, where $\cS_\infty$ is defined in Definition \ref{Smdef}. Since $\cS_\infty$ is a constant-coefficients operator, by classical symbol analysis we deduce that

\begin{prop} \label{prop: spectrum infty} For $f$ as in \eqref{def:f} we have
\[
\sigma_e(\cS_\infty) \cap \Sigma = \bigg\{ \omega \in \C : f(\omega)= t, t \in \sigma_e(\curl \curl_0|_{\ker(\Div)}) \bigg\} \cap \Sigma
\]
\end{prop}
\begin{proof}
Without loss of generality, we may assume $\omega \in \Sigma_1$, the case $\omega \in \Sigma_2$ being similar. We first note that if $\omega \in \Sigma_1$, $\Theta_{m, \infty}(\omega)$ is (boundedly) invertible. Studying the spectrum of $\cS_\infty$ is then equivalent to
considering directly the pencil $(\curl \curl_0 - f(\omega))|_{\ker(\Div)}$, which we will call again $\cS_\infty(\omega)$ with an abuse of notation. Now, $\omega \in \sigma_e(\cS_{\infty}) \cap \Sigma_1$ iff there exists a Weyl singular sequence
$u_n\in H(\Div0,\Omega)$ such that  $\curl \curl_0 u_n - f(\omega)u_n \to 0$, which holds iff $f(\omega) \in \R$ and $f(\omega) \in \sigma_e(\curl \curl_0|_{\ker(\Div)})$. The claim is proved.
\end{proof}

\begin{prop} \label{prop: We infty}
\[
W_e(\cS_\infty) \cap \Sigma= \big\{ \omega \in \C : f(\omega) = t, t \in W_e(\curl \curl_0|_{\ker(\Div)}) \big\} \cap \Sigma
\]
\end{prop}
\begin{proof}
If $(u_n)$ is a Weyl sequence in $H(\Div0, \Omega)$ with $\| u_n \| = 1$ for all $n$, then $(\curl_0 u_n, \curl_0 u_n) - t \to 0$ if and only if $(\curl_0 u_n, \curl_0 u_n) - f(\omega) \to 0$ for all $\omega\in\C$ such that $f(\omega) = t$.
\end{proof}

As a consequence of Proposition \ref{prop: spectrum infty} we can study the asymptotics of the spectrum of $\cS_\infty$.

\begin{prop} \label{prop: asymptotics spectrum} Assume that
\begin{equation}\label{eq: asymp}
f(\omega_n) =  \frac{\Theta_{e,\infty}(\omega_n)\Theta_{m, \infty}(\omega_n)}{(i\gamma_e + \omega_n)(i \gamma_m + \omega_n)} = t_n \to + \infty
\end{equation}
as $n \to \infty$, $t_n \in \sigma_e(\curl \curl_0|_{\ker(\Div)})$. Then the following are true:
\begin{enumerate}[label = (\roman*)]
\item if $\re \omega_n$ is bounded as $n \to \infty$ then $\dist(\im \omega_n, \{-i \gamma_e, -i \gamma_m\}) \to 0$. If $\omega_n \to -i \gamma_x$, $x=e,m$ then $\omega_n = - i \gamma_x - \frac{i c}{t_n} + o(1/t_n)$ as $n \to \infty$, where $c \in \R \setminus \{0\}$ is explicitly given by $c = (\theta^0_x)^2\big( - \gamma_x + \frac{(\theta^0_x)^2 (\theta^0_y)^2}{\gamma_y - \gamma_x} \big)$, with $y \neq x$, $x,y \in \{ e, m\}$.
\item If $|\re \omega_n| \to +\infty$ then $t_n = |\re \omega_n|^2 + o(|\re \omega_n|^2)$ as $n \to \infty$ and $\im \omega_n \to 0$ with the asymptotics $\im \omega_n = - \frac{1}{\re \omega_n^2} ((\theta^0_e)^2 \gamma_e + (\theta^0_m)^2 \gamma_m) + o(1/\re \omega_n^2)$ .
\end{enumerate}
\end{prop}
\begin{proof}
(i) $(\im \omega_n)_n$ is a bounded sequence since the numerical range of $S_{1, \infty}$ is contained in an horizontal strip. Since also $(\re \omega_n)_n$ is a bounded sequence by assumption, we may assume that up to a subsequence $\omega_n \to \omega_{\infty} \in \C$. From \eqref{eq: asymp} we have that
\[
(\omega_n^2 + i \gamma_m \omega_n - (\theta_m^0)^2)(\omega_n^2 + i \gamma_e \omega_n - (\theta_e^0)^2) = t_n (\omega_n + i \gamma_e)(\omega_n + i \gamma_m)
\]
so
\[
\limsup_{n \to \infty} |t_n (\omega_n + i \gamma_e)(\omega_n + i \gamma_m)| \leq |(\omega_\infty^2 + i \gamma_m \omega_\infty - (\theta_m^0)^2)(\omega_\infty^2 + i \gamma_e \omega_\infty - (\theta_e^0)^2)|
\]
which implies that either $\omega_\infty = - i \gamma_e$ or $\omega_\infty = -i \gamma_m$. Set then $\omega_n = -i \gamma_e + \eps_n$ as $n \to \infty$, where $\eps_n \to 0$, $\eps_n \in \C$. Substituting this ansatz in \eqref{eq: asymp} and keeping only the zeroth order terms we get
\[
- (\theta_e^0)^2(- \gamma_e^2 + \gamma_e \gamma_m - (\theta_m^0)^2) = z (i (\gamma_m - \gamma_e)), \quad z = \lim_{n \to \infty} t_n \eps_n
\]
We note that the existence of the limit for $z$ is up to a subsequence, and can be easily deduced from the fact that equation \eqref{eq: asymp} is not satisfied if $(t_n \eps_n)_n$ is not a bounded sequence. Hence we have that
\[
\eps_n = \frac{z}{t_n} + o(1/t_n) = - \frac{i (\theta_e^0)^2}{t_n} \bigg( - \gamma_e + \frac{(\theta_e^0)^2 (\theta_m^0)^2}{\gamma_m - \gamma_e} \bigg) + o(1/t_n)
\]
as $n \to \infty$, concluding the proof of $(i)$. \\
(ii) To shorten the notation, let us set $x_n = \re \omega_n$, $y_n = \im \omega_n$. From equation \eqref{eq: asymp}, recalling that $(y_n)_n$ is bounded, after taking the real part we see that
\[
x_n^4 - t_n x_n^2 = o(x_n^4) \quad \Rightarrow \quad t_n = x_n^2 + o(x_n^2), \quad n \to \infty.
\]
A further inspection of equation \eqref{eq: asymp} gives that the $o(x_n^2)$-term must be in the form $c_n + o(1)$, where $c_n$ possibly depends on $y_n$. Using the ansatz $t_n = x_n^2 +c_n$ in the real part of \eqref{eq: asymp} and neglecting the lower order terms we get
\begin{multline*}
x_n^4 - 6 x_n^2 y_n^2 - 3 x_n^2 y_n (\gamma_e + \gamma_m) + x_n^2 (- (\theta_e^0)^2 - (\theta_m^0)^2 - \gamma_e \gamma_m) \\
= x_n^4+ x_n^2(c_n - y_n^2 - y_n(\gamma_e + \gamma_m) - \gamma_e \gamma_m)
\end{multline*}
from which we deduce $-5 y_n^2 - 2 y_n ( \gamma_e + \gamma_m) - (\theta_e^0)^2 - (\theta_m^0)^2 = c_n + o(1)$ as $n \to \infty$. In order to continue the analysis we now turn to the imaginary part of \eqref{eq: asymp}. More explicitly, we have
\begin{multline} \label{eq: asymp eq imag}
4 x_n^3 y_n - 4 x_n y_n^3 + (x_n^3 - 3 x_n y_n^2)(\gamma_e + \gamma_m) + 2x_n y_n (- (\theta^0_e)^2 - (\theta_m^0)^2 - \gamma_e \gamma_m) \\
- x_n \big(\gamma_m (\theta_e^0)^2 + \gamma_e (\theta_m^0)^2\big) = (x_n^2 + c_n) x_n( 2 y_n + (\gamma_e + \gamma_m))
\end{multline}
The term in $x_n^3$ simplifies. Now, if $y_n$ does not tend to zero as $n \to \infty$, the highest order term in the previous equation is $x_n^3 y_n$, so we get the equation $2 x_n^3 y_n = o(x_n^3)$, which is a contradiction. Hence $y_n \to 0$. The candidate highest order terms are $x_n^3 y_n$ and $x_n$. By direct inspection one checks that if $x_n^3 y_n = o(x_n)$ or $x_n = o(x_n^3 y_n)$ as $n \to \infty$ the previous equation gives a contradiction. So it must be $y_n = \frac{z_n}{x_n^2} + o(1/x_n^2)$ as $n \to \infty$. By using this ansatz in \eqref{eq: asymp eq imag} and keeping only the highest order term (namely the ones in $x_n$), it may be proved that
\[
- x_n (\gamma_m (\theta_e^0)^2 + \gamma_e (\theta_m^0)^2) + 4 z_n x_n = 2 z_n x_n + c_n(\gamma_e + \gamma_m) x_n,
\]
for $c_n = - (\theta_e^0)^2 - (\theta_m^0)^2  + o(1)$. Thus, $z_n = - \frac{1}{2} ((\theta_e^0)^2 \gamma_e + (\theta_m^0)^2 \gamma_m) + o(1)$ and $y_n = - \frac{1}{2 x_n^2} ((\theta_e^0)^2 \gamma_e + (\theta_m^0)^2 \gamma_m) + o(1/x_n^2)$.
\end{proof}

We now turn to the bounded pencil $G(\cdot)$ defined in \eqref{B1def}. Recall that  $V_e(\omega)(x) = \frac{\Theta_e(\omega,x)}{(\omega + i \gamma_e)(\omega + i \gamma_m)}$, $\omega \in \Sigma$, $x \in \Omega$.

\begin{prop}
Assume that $\theta_e$ is a continuous function in $\overline{\Omega}$. Then
\[
\begin{split}
\sigma_e(G) &= \{ \omega \in \C :\, \exists x_0 \in \Omega, \, \Theta_e(\omega,x_0) = 0 \}\\
&= \{ \omega \in \C :\, \exists x_0 \in \Omega,\, \omega = - i \gamma_e/2 \pm (\sqrt{- \gamma_e^2 + 4 \theta_e^2(x_0)})/2 \}
\end{split}
\]
\end{prop}
\begin{proof}
Since $(\omega + i \gamma_e) (\omega + i \gamma_m)$ is constant in $x \in \Omega$, we can replace $V_e$ by $\Theta_e$ in the definition of $G$ without changing the spectrum. We first note that the set
\[\{ \omega \in \C : \re(\Theta_e(\omega, x_0)) = 0, \im(\Theta(\omega, x_1)) = 0, \, x_0 \neq x_1\}\]
actually coincides with  $\{ \omega \in \C : \exists x_0 \in \Omega : \Theta_e(\omega,x) = 0 \}$. In fact, it is easy to check that $\im(\Theta(\omega, x)) = 2 \re \omega \im \omega + \re \omega \gamma_e$ does not depend on $x \in \Omega$, and therefore $\im(\Theta(\omega, x_1)) = 0$ for $x_1 \in \Omega$ if and only if $\im(\Theta(\omega, x)) = 0$ for all $x \in \Omega$.\\
It is clear that if either $\re(\Theta_e(\omega))$ or $\im(\Theta_e(\omega))$ is strictly positive (or strictly negative) in the whole of $\Omega$ then by the Lax-Milgram theorem the problem
\[(\Theta_e(\omega) \nabla u, \nabla v) = \langle F , v \rangle, \:\:F \in H^{-1}(\Omega),\: u,v \in \dot{H}^1_0(\Omega)\]
has a unique solution $u_F \in \dot{H}^1_0(\Omega)$. This proves the inclusion
\[
\sigma_e(G(\omega)) \subseteq \{ \omega \in \C : \exists x_0 \in \Omega : \Theta_e(\omega,x) = 0 \}
\]
The reverse inclusion (which uses the continuity of $\theta_e$) follows by constructing quasi-modes as in the proof of \cite[Proposition 27]{MR3942228}.
\end{proof}

\section{Limiting essential spectrum and spectral pollution}
\label{sec:lim_ess_spec}
The aim of this section is to enclose the set of spectral pollution for the domain truncation method applied to the Drude-Lorentz pencil $\cL$.
We start by recalling three definitions.

\begin{definition} \label{lim-sigma-app}
For a family of operator-valued functions $(F_n(\cdot))$ defined on some set $K\subset \C$, the {\it limiting approximate point spectrum},
denoted $\sigma_{\rm app}((F_n)_n)$, is the set of $\omega \in K$ such that there exists a sequence $(u_n)_n$ with
$u_n\in \dom(F_n(\omega))$ for each $n$, $\| u_n\| = 1$, as $n\to\infty$ and $\| F_n(\omega)u_n\| \to 0$ as $n\to \infty$.
For a family of operators $(F_n)_n$, one takes $K=\C$ and the requirement is that $\|(F_n-\omega I)u_n\|\to 0$ as $n\to\infty$.
\end{definition}

\begin{definition} \label{lim-sigma-ess}
For a family of operator-valued functions $(F_n(\cdot))$ defined on some set $K\subset \C$, the {\it limiting essential spectrum},
denoted $\sigma_e((F_n)_n)$, is the subset of $\sigma_{\rm app}((F_n)_n)$ consisting of $\omega \in K$ for which the sequence $(u_n)_n$ as in the definition \ref{lim-sigma-app} has the additional property $u_n \rightharpoonup 0$, $n \to \infty$.
\end{definition}

\begin{definition}\label{region-of-boundedness}
For an family of operator-valued functions $(F_n(\cdot))$ defined on some set $K\subset \C$, the {\it region of boundedness},
denoted $\Delta_b((F_n)_n)$, is the set of $\omega \in K$ such that $(F_n(\omega))^{-1}$ exists for all sufficiently large $n$ and
$\limsup_{n\to\infty} \| (F_n(\omega))^{-1} \| < +\infty$. For a family of operators $(F_n)_n$, then
$\Delta_b((F_n)_n)$, is the set of $\omega \in \C$ such that $(F_n-\omega I)^{-1}$ exists for all sufficiently large $n$ and
$\limsup_{n\to\infty} \| (F_n - \omega I)^{-1} \| < +\infty$.
\end{definition}

Given an unbounded Lipschitz domain $\Omega$, let $(\Omega_n)_n$ be a monotonically increasing sequence of Lipschitz bounded domains exhausting $\Omega$. Note that we do not make any assumption on the topology of $\Omega$; in particular, $\R^3 \setminus \overline{\Omega}$ may have infinitely many connected components.
We will denote by $\cA_n$, $\cL_n$, $\cS_{1,n}$, \emph{etc.,} the operators or pencils obtained from $\cA$, $\cL$, $\cS_1$ by replacing $\Omega$ with $\Omega_n$ in their domain definitions. In order to avoid cumbersome notation we will not use the subscript $n$ to denote the restriction of multiplication operators to $\Omega_n$.\\
The domain truncation method consists of finding $\sigma(\cL_n)$, $n \in \N$, in the hope that for sufficiently large $n$
 these will be good approximations to $\sigma(\cL)$. Ideally, one would like to prove that $(\cL_n)_n$ is a spectrally exact approximation of $\cL$, that is for every $\omega \in \sigma(\cL)$ there exists $\omega_n \in \sigma(\cL_n)$, $n \in \N$ such that $\omega_n \to \omega$; and conversely, every limit point of sequences $(\omega_n)_n$ with $\omega_n \in \sigma(\cL_n)$, $n \in \N$, lie in $\sigma(\cL)$. However, this is false in general. Indeed, spectral pollution may appear due to the non-self-adjointness of the operators involved. Therefore, our strategy will be to enclose the set of spectral pollution in a (possibly) small subset of $\C$ and to show that we can approximate exactly the discrete points of $\sigma(\cL)$ outside the set of spectral pollution.
%
%

We begin with a result about the generalised resolvent convergence of the operators involved.

\begin{theorem} \label{thm:gsr} The following statements hold.
\begin{enumerate}[label=(\roman*)]
\item $\cA_n \gsr \cA$, $n \to \infty$.
\item $\cL_n(\cdot) \gsr \cL(\cdot)$ for all $\omega \in (\Delta_b((\cL_n)_n)\setminus \{-i \gamma_e, - i \gamma_m\}) \cap \rho(\cL)$, $n \to \infty$
\item  $\Delta_b((\cL_n)_n)\cap \Sigma = \Delta_b((\cS_{1,n})_n)\cap \Sigma$
\item $\cS_{1,n}(\cdot) \gsr \cS_{1}(\cdot)$, for all $\omega \in (\Delta_b((\cL_n)_n) \cap \Sigma) \cap \rho(\cL)$, $n \to \infty$.
\end{enumerate}
\end{theorem}
\begin{proof}
(i) According to Proposition \ref{prop: num range 1}, $W(\cA_n) \subset \R \times [-i\gamma_e, 0]$ for all $n \in \N$, and the same enclosure holds for the numerical range of $\cA$ as well. We then deduce that $\norma{(\cA_n - \la)^{-1}} \leq \dist(\la, W(\cA_n))^{-1} \leq (\min \{|\im \la|, |\im \la + \gamma_e| \}^{-1})$ for all $\la \in (\R \times [-i\gamma_e, 0])^c$. In particular, $\Delta_b((\cA_n)_n) \cap \rho(\cA) \neq \emptyset$.
The matrices $\cA_n$, $n \in \N$ and $\cA$ are diagonally dominant of order $0$ for all $n$ since the operators $B_n$, $n \in \N$ in the matrix representation \eqref{def:cA} are bounded for all $n$.
Let $P_n$ denote projection from $L^2(\Omega)^3$ to $L^2(\Omega_n)^3$ by restriction, so that $P_n \sto I_{L^2(\Omega)^3}$ as $n \to \infty$.
It is clear that $B_n = B P_n$ converges strongly to $B$, and similarly $D_n = D P_n \to D$ strongly.
Due to \cite[Thm 3.1]{MR3694623}, to conclude that $\cA_n \gsr \cA$ it is enough to show that there exists a core $\Phi$ of $H_0(\curl, \Omega) \oplus H(\curl, \Omega)$ such that $\norma{A_n P_n u - A u} \to 0$ for all $u \in \Phi$. This last property is satisfied by  $C^\infty_c(\Omega)^3 \oplus C^\infty_c(\overline{\Omega})^3$ because $\curl_0 P_n \varphi = \curl_0 \varphi$ for all $\varphi \in C^\infty_c(\Omega)^3$ and sufficiently large $n$, and $\curl P_n \psi = P_n \curl \psi$ in $\Omega_n$ for all $\psi \in C^\infty_c(\overline{\Omega})^3$.\\
(ii) To prove $\cL_n(\omega) \gsr \cL(\omega)$ for $\omega \in (\Delta_b((\cL_n)_n) \cap \rho(\cL)) \setminus \{-i \gamma_m, -i \gamma_e \}$, $n \in \N$ it is sufficient to write
{\small
\[
(\cA_n - \omega)^{-1} \!=\! \begin{pmatrix}
\cL_n(\omega)^{-1} & - \cL_n(\omega)^{-1} B(-iD - \omega)^{-1} \\
(-iD - \omega)^{-1} B \cL_n(\omega)^{-1} & (-iD - \omega)^{-1} (I + B \cL_n(\omega)^{-1} B (-iD - \omega)^{-1})
\end{pmatrix}
\]
}
hence, since $(\cA_n - \omega)^{-1} (P_n F, 0)^t \to (\cA - \omega)^{-1} (F, 0)^t $ for all $F \in L^2(\Omega)^6 \oplus L^2(\Omega)^6$, we deduce that $\cL_n(\omega)^{-1}P_n F \to \cL(\omega)^{-1}F$, $n \to \infty$. \\
(iii) Let $\omega \in (\Delta_b((\cL_n)_n)\cap \Sigma)$, so that $\sup_{n \geq n_0} \norma{(\cL_n(\omega))^{-1}} < \infty$. The identity
{\small
\begin{multline}
\label{eq:res_eq_LS}
(\cL_n(\omega))^{-1} \\
=\! \begin{pmatrix}
\cS_{1,n}(\omega)^{-1} & - \cS_{1,n}(\omega)^{-1} i \curl (V_m(\omega))^{-1} \\
(V_m(\omega))^{-1} i\curl_0 \cS_{1,n}(\omega)^{-1} & (V_m(\omega))^{-1} (I + \curl_0 \cS_{1,n}(\omega)^{-1} \curl (V_m(\omega))^{-1})
\end{pmatrix}
\end{multline}}
implies that $\norma{\cL_n(\omega)^{-1}} \geq \norma{\cS_{1,n}(\omega)^{-1}}$, $n \geq n_0$, hence $\sup_{n \geq n_0}\norma{(\cS_{1,n}(\omega))^{-1}} < \infty$; equivalently, $\omega \in (\Delta_b((\cS_{1,n})_n)\cap \Sigma)$. \\
Conversely, if $\omega \in (\Delta_b((\cS_{1,n})_n)\cap \Sigma)$ then, by definition of region of boundedness, there exists $n_0 \in \N$ and $C > 0$ such that $\omega \in \rho(\cS_{1,n})$ for $n \geq n_0$ and $\sup_{n \geq n_0} \norma{\cS_{1,n}(\omega)^{-1}} \leq C$. Hence, if $f \in L^2(\Omega)^3$, the equation $\cS_{1,n}(\omega) u_n = P_n f$ has a unique solution $u_n \in L^2(\Omega_n)^3$ with the uniform \emph{a priori} bound $\norma{u_n}_{L^2(\Omega_n)^3} \leq C \norma{f}_{L^2(\Omega)^3}$. This implies that
\[
|\langle \Theta_m(\omega)^{-1} \curl_0 u_n, \curl_0 u_n \rangle | \leq |\langle P_n f, u_n \rangle | + |\langle V_e(\omega) u_n, u_n \rangle |
\]
hence
\[
c_1(\omega) \norma{\curl_0 u_n}^2 \leq \frac{\norma{f}^2}{4 \delta} + (\delta + \norma{V_e(\omega)}) \norma{u_n}^2 \leq [C^2 (\delta + \norma{V_e(\omega)}) + 1/4 \delta] \norma{f}^2
\]
and all the constants appearing in the previous estimate are independent of $n \geq n_0$, so in particular
\[\sup_{n \geq n_0} \norma{\curl_0\cS_{1,n}(\omega)^{-1}} \leq C'(\omega), \]
where we can set for example $C' = c_1(\omega)^{-1} [C^2 (1 + \norma{V_e(\omega)}) + 1/4]$.
Now we can repeat the previous estimate starting from elements $f = \curl g \in \curl L^2$, where we note that $\langle P_n \curl g, u_n \rangle = \langle g, \curl_0 u_n \rangle$ can be estimated in term of $\curl_0 u_n$ which is uniformly bounded in terms of the datum by the previous discussion. Altogether we obtain that there exists a constant $C''$ depending on $C'$ and $\omega$ such that
\[
\sup_{n \geq n_0} \norma{\curl\, \cS_{1,n}(\omega)^{-1}\, \curl_0} \leq C''
\]
Now the claim of the theorem follow from \eqref{eq:res_eq_LS} since the right-hand side therein is uniformly bounded in $n$, $n \geq n_0$. \\
(iv) Finally, $\cS_{1,n}(\omega)^{-1} P_n \sto \cS_1(\omega)^{-1}$ for all $\omega \in (\Delta_b((\cL_n)_n) \cap \Sigma)$ follows by observing that $\omega \in \Delta_b((\cL_n)_n) \cap \Sigma $, hence $(ii)$ implies that $\cL_n(\omega)^{-1}P_n \sto \cL(\omega)^{-1}$; therefore \eqref{eq:res_eq_LS} implies that $\cS_{1,n}(\omega)^{-1}P_n \sto \cS_1(\omega)$ (by direct calculation on vectors $(f,0)^t$).

\end{proof}

\begin{corollary} \label{cor: sigmae An}
$\sigma_{\rm poll}((\cA_n)_n) \subset \sigma_e((\cA_n)_n)$.
\end{corollary}
\begin{proof}
From Definition \eqref{def:cA}, $\cA_n$ is seen to be $\cJ$-self-adjoint with respect to $\cJ = \diag(i , -i, i, -i) J$, with $J u = \bar{u}$ the componentwise complex conjugation. The result is then a consequence of Thm 2.3 in \cite{MR3831156}.
\end{proof}

%
%

%

\begin{prop} \label{prop: pollution Ln}
$\sigma_{\rm poll} ((\cL_n)_n) = \sigma_{\rm poll}((\cA_n)_n) \setminus \{-i \gamma_e, -i \gamma_m \}$
\end{prop}
\begin{proof}
\cite[Thm. 2.3.3(ii)]{TreB} implies that, for fixed $n$, $\sigma(\cL_n) = \sigma(\cA_n) \setminus \{-i \gamma_e, -i \gamma_m \}$ and similarly $\sigma(\cL) = \sigma(\cA) \setminus \{-i \gamma_e, -i \gamma_m \}$. Hence, $\sigma_{\rm poll} ((\cL_n)_n) \subset \sigma_{\rm poll}((\cA_n)_n) \setminus \{-i \gamma_e, -i \gamma_m \}$. Conversely, one may observe that if $\la_n \in \sigma(\cA_n)$, $\la_n \to \la \in (\rho(\cA) \setminus  \{-i \gamma_e, -i \gamma_m \})$, then for big enough $n$ $\la_n \notin \{-i \gamma_e, -i \gamma_m \}$, so $\la_n \in \sigma(\cL_n)$ and $\la_n \to \la \in \rho(\cL)$. Thus, $\la \in \sigma_{\rm poll}((\cL_n)_n)$.
\end{proof}

\begin{rem}
The poles $\{-i \gamma_e, -i \gamma_m \}$ may or may not be in the essential spectrum of $\cA$. If $B$ is compactly supported and $A$ restricted to divergence-free vector field has compact resolvent, then the poles belong to $\sigma_e(\cA)$, see \cite[Prop. 2.2]{MR3543766} for a proof in a similar setting. However, for the purposes of this paper we are not interested in the poles that are out of the domain of definition of $\cL$.
\end{rem}

%

\begin{prop}
$\sigma_e((\cA_n)_n) \setminus \{-i \gamma_e, -i \gamma_m \}  = \sigma_e((\cL_n)_n) \setminus \{-i \gamma_e, -i \gamma_m \}$
\end{prop}
\begin{proof}
$\cA_n$ is diagonally dominant for every $n$ and the norms of the off-diagonal entries do not depend on $n$;
furthermore, $\Delta_b((D_n)_n = \C \setminus \{-i \gamma_e, -i \gamma_m \}$. The result
therefore follows from \cite[Proposition 2.3.4(i),(iii)]{SThesis}.
\end{proof}

\begin{prop} \label{prop:sigma-app-Ln} The following identities hold.
\begin{enumerate}[label=(\roman*)]
\item $\sigma_e((\cL_n)_n) \cup \sigma_p(\cL) = \sigma_{\rm app}((\cL_n)_n)$;
\item $(\sigma_e((\cS_{1,n})_n) \cup  \sigma_p(\cS_1)) \cap \Sigma = \sigma_{\rm app}((\cS_{1,n})_n) \cap \Sigma$.
\end{enumerate}
\end{prop}
\begin{proof}
The proof is a generalisation of \cite[Prop. 2.15(ii)]{MR3831156} to families of operators.\\
Observe that the inclusion $\sigma_p(\cL) \subset \sigma_{\rm app}((\cL_n)_n)$ (and the analogous inclusion for $(\cS_{1,n})_n)$) are consequence of the gsr convergence of $\cL_n$ to $\cL$ established in Theorem \ref{thm:gsr}. Indeed, we claim that if $\cL_n(\omega_0) \gsr \cL(\omega_0)$ for some $\omega_0 \in \Delta_b((\cL_n)_n) \cap \rho(\cL)$, then for all $u \in \dom(\cL)$ there exists a sequence $u_n \in \dom(\cL_n)$, $\norma{u_n} = 1$ such that $\norma{u_n - u} \to 0$, $\norma{\cL_n(\omega)u_n - \cL(\omega)u} \to 0$, $n \to \infty$, $\omega \in \C \setminus \{-i \gamma_e, -i \gamma_m\}$. Assuming the claim is satisfied, if $\omega \in \sigma_p(\cL)$ with eigenfunction $u$, then there exists an approximating sequence $(u_n)_n$ as above, and therefore $\sigma_p(\cL) \subset \sigma_{\rm app}((\cL_n)_n)$.\\
To prove the claim, one first realises that if $\omega \in \Delta_b((\cL_n)_n) \cap \rho(\cL)$ then the sequence $u_n := \cL_n(\omega)^{-1} P_n \cL(\omega) u$ has the required properties. In the general case, one first note that for $t < 0$ sufficiently big $\omega \in \Delta_b((\cL_n(\cdot) + i t)) \cap \rho(\cL(\cdot) + it)$, see the proof of Lemma \ref{lemma:bddcls}. Therefore, given $u \in \dom(\cL + it) =  \dom(\cL)$, there exists a sequence $u_n \in \dom( \cL_n(\cdot) + i t) = \dom(\cL_n)$ such that $\norma{u_n - u} \to 0$ and $\norma{(\cL_n(\omega) + it) u_n - (\cL(\omega) + it) u} \to 0$, and therefore $\norma{ \cL_n(\omega) u_n - \cL(\omega) u} \to 0$, as claimed.\\
The inclusion "$\subset$" in $(i)$ and $(ii)$ is immediate from Definitions \ref{lim-sigma-app}, \ref{lim-sigma-ess}, and the previous observation.\\
We now prove that $\sigma_e((\cL_n)_n) \cup \sigma_p(\cL) \supset \sigma_{\rm app}((\cL_n)_n)$. Let $\omega \in \sigma_{\rm app}((\cL_n)_n)$, that is, there exists a sequence of elements $u_n \in \dom(\cL_n)$, $\norma{u_n} = 1$, $n \in \N$, such that $\norma{\cL_n(\omega)u_n} \to 0$. Since the unit ball in a Hilbert space is weakly compact, we may assume that, up to a subsequence, $u_n \rightharpoonup u$ in $L^2(\Omega)^3 \oplus L^2(\Omega)^3$. If $u = 0$, then $\omega \in \sigma_e((\cL_n)_n)$, concluding the proof. Assume then $u \neq 0$. Theorem \ref{thm:gsr}(iii) implies that there exists $\omega_0 \in \Delta_b((\cL_n)_n) \cap \rho(\cL)$ such that $\cL_n(\omega_0) \gsr \cL(\omega_0)$. Now,
\begin{equation}
\label{eq:cL_n(omega_0)}
\cL_n(\omega_0)u_n = (\cL_n(\omega_0) - \cL_n(\omega))u_n + \cL_n(\omega) u_n
\end{equation}
and we notice that $\cL_n(\omega_0) - \cL_n(\omega) := \cB(\omega, \omega_0) \diag(P_n, P_n)$, where $\cB$ is a bounded $2 \times 2$ block operator matrix not depending on $n$. Taking the inverse of $\cL_n(\omega_0)$ in \eqref{eq:cL_n(omega_0)} gives
\[
u_n = \cL_n(\omega_0)^{-1} \cB\,\, \diag(P_n, P_n) u_n + \eps_n, \quad \eps_n:= \cL_n(\omega_0)^{-1} \cL_n(\omega)u_n.
\]
Now, $\eps_n\to 0$, $n \to \infty$ because $\omega_0 \in \Delta_b((\cL_n)_n)$ and $\cL_n(\omega) u_n \to 0$, $n \to \infty$. Moreover, the weak convergence $u_n \rightharpoonup u$, the strong convergence $\cL_n(\omega_0)^{-1} \sto \cL(\omega_0)^{-1}$ and the uniqueness of the weak limit imply that
\[
u = \cL(\omega_0)^{-1} \cB\, u,
\]
or equivalently, since $\cB = \cL(\omega_0) - \cL(\omega)$, that $\cL(\omega) u = 0$. Thus, $\omega \in \sigma_p(\cL)$.\\
The proof of $(\sigma_e((\cS_{1,n})_n) \cup  \sigma_p(\cS_1)) \cap \Sigma \supset \sigma_{\rm app}((\cS_{1,n})_n) \cap \Sigma$ in $(ii)$ is analogous to the proof of the same inclusion in $(i)$ with $\cS_{1,n}(\omega)$ replacing $\cL_n(\omega)$; in place of $\cS_{1,n}(\omega_0)$ for some $\omega_0 \in \Delta_b((\cS_{1,n})_n)$ we choose $\cS_{1,n}(\omega) + i t$ for $t < 0$ big enough. Notice that this is possible since $\omega \in \Delta_b((\cS_{1,n}(\cdot) + it)_n)$ for $t<0$ large enough, see proof of Lemma \ref{lemma:bddcls}. The proof is concluded.
\end{proof}

\begin{theorem}
$\sigma_e((\cL_n)_n) \cap \Sigma = \sigma_e ((\cS_{1,n})_n) \cap \Sigma$
\end{theorem}
\begin{proof}
Proposition \ref{prop:sigma-app-Ln} implies
\[
\sigma_e((\cL_n)_n) \cup \sigma_p(\cL) = \sigma_{\rm app}((\cL_n)_n), \quad \sigma_e((\cS_{1,n})_n) \cup  \sigma_p(\cS_1)= \sigma_{\rm app}((\cS_{1,n})_n)
\]
Now, the definitions of region of boundedness and of limiting approximate point spectrum, together with the equality $\sigma_{\rm app}((\cL_n)_n) = \sigma_{\rm app}((\cL^*_n)_n)^*$, imply that $\sigma_{\rm app}((\cL_n)_n) = \C \setminus \Delta_b((\cL_n)_n)$ and $\sigma_{\rm app}((\cS_{1,n})_n) = \C \setminus \Delta_b((\cS_{1,n})_n)$. By Thm. \ref{sigmaLS}, $\sigma_p(\cL) \cap \Sigma = \sigma_p(\cS_1) \cap \Sigma$; and as a consequence of Thm. \ref{thm:gsr}(iii), $\Delta_b((\cS_{1,n})_n) \cap \Sigma = \Delta_b((\cL_n)_n) \cap \Sigma$. Thus, up to intersection with $\Sigma$ we have
\[
\begin{split}
\sigma_e((\cL_n)_n) \cup \sigma_p(\cL) &= \sigma_{\rm app}((\cL_n)_n) = \C \setminus \Delta_b((\cL_n)_n) = \C \setminus \Delta_b((\cS_{1,n})_n) \\
&= \sigma_{\rm app}((\cS_{1,n})_n) = \sigma_e((\cS_{1,n})_n) \cup \sigma_p(\cS_1).\qedhere
\end{split}
\]
\end{proof}

In the following proposition we use the notion of \emph{discrete compactness} for sequences of operators in varying Hilbert spaces. We refer to \cite[Def. 2.5]{MR3694623} and the references therein for the relevant definitions and properties.

\begin{prop} \label{prop: decomp sigmaeS1n}
Given $n \in \N$, the following equality holds.
\[\sigma_e((\cS_{1,n})_n) \cap \Sigma = \left(\sigma_e((\cS_{\infty, n})_n) \cup (\, \sigma_e((G_{n})_n)\right) \cap \Sigma\,),\]
where $G_n$ is defined as in \eqref{B1def}.
\end{prop}
\begin{proof}
This can be proved along the lines of \cite[Section 7, Section 8]{BFMT}. For the sake of completeness we recall here the main steps of the proof.\\
Observe that Proposition \ref{thm: compactness} implies that $M_{e,n}(\omega) = (\Theta_e(\omega) - \Theta_{e, \infty}(\omega))P^n_{\ker(\Div)}$ is a compact operator from $H(\curl, \Omega_n)$ to $L^2(\Omega_n)^3$ for every $n$.
According to the decomposition of the coefficients \eqref{eq:coeffs-infty}, up to an operator which is vanishing uniformly in $n$, the sequence $M_{e,n}(\omega)$ is compactly supported in $\Omega_n$ for each $n$ and equals $(\Theta_e(\omega) - \Theta_{e, \infty}(\omega))\chi_{\Omega_R \cap \Omega_n}P^n_{\ker(\Div)}$, where $\Omega_R = \Omega \cap B(0,R)$ contains the compact support of $\Theta_e(\omega) - \Theta_{e, \infty}(\omega)$. This last sequence of operators is clearly discretely compact from $H(\curl, \Omega_n)$ to $L^2(\Omega_n)^3$ because of the compact embedding of $H(\curl, \Omega_R) \cap H(\Div0, \Omega_R)$ into $L^2(\Omega_R)$.
Since discretely compact perturbations do not modify the limiting essential spectrum, $\sigma_e(((\cS_{1,n})_n) \cap \Sigma = \sigma_e((\cS_{1, n} + M_{e,n}))_n) \cap \Sigma$.
Now we note that
\begin{multline*}
\cS_{1, n}(\omega) + M_{e,n}(\omega) \\
= \curl \Theta_m(\omega)^{-1} \curl_0 - \frac{\Theta_{e, \infty}(\omega)}{(\omega + i \gamma_e) (\omega + i \gamma_m)}P_{\ker(\Div)} -  \frac{\Theta_{e}(\omega)}{(\omega + i \gamma_e) (\omega + i \gamma_m)}P_{\nabla}
\end{multline*}
has a triangular block operator matrix representation with respect to the Helmholtz decomposition $\nabla H^1_0(\Omega_n) \oplus H(\Div0, \Omega_n)$. More specifically, if we define $\cS_{m,n}(\omega)$, as $\cS_m$ with $\Omega_n$ replacing $\Omega$ in the domain definition, we have
\[
\cS_{1,n}(\omega) + M_{e,n}(\omega) \simeq \cT_n(\omega) = \begin{pmatrix}
-P^n_\nabla V_e(\omega)P^n_\nabla & 0 \\
-P^n_{\ker(\Div)}V_e(\omega)P^n_\nabla & \cS_{m,n}
\end{pmatrix}.
\]
Now since $\cT_n$ is a sequence of triangular block operator matrices with bounded off-diagonal entries and with $J$-selfadjoint diagonal entries, Theorem \ref{thm: ess spec} implies that
\[
\sigma_e((\cT_n)_n) \cap \Sigma = (\sigma_e((-P^n_\nabla V_e(\omega)P^n_\nabla)_n)_n  \cup (\sigma_e((\cS_{m,n})_n) \cap \Sigma)
\]
Now \cite[Proposition 7.3]{BFMT} implies that $\sigma_e((\cS_{m,n})_n) = \sigma_e((\cS_{\infty,n})_n)$. We give here a sketch of the proof. The proof is modelled upon the proof of Proposition \ref{thm: difference res}. The idea is to establish that for a suitably chosen $\omega \in \Sigma$ the difference $\cK_n(\omega) = \cS_{m,n}(\omega)^{-1} -  \cS_{\infty,n}(\omega)^{-1}$ is discretely compact and that $\cK_n(\omega)^*P_n$ is strongly convergent. Then the equality of the limiting essential spectra follows from \cite[Thm. 2.12(ii)]{MR3831156}. The strong convergence
\[\cK_n(\omega)^*P_n = \cS_{m,n}(\omega)^{-*}P_n -  \cS_{\infty,n}(\omega)^{-*}P_n \sto \cS_m(\omega)^{-*} - \cS_{\infty}(\omega)^{-*}\]
for $\omega \in \Delta_b((\cK_n)_n) \cap \rho(\cS_m) \cap \rho(\cS_\infty)$ can be proved along the lines of Thm.\ref{thm:gsr}(iv). \\
It remains to prove that $(\cK_n(\omega))_n$ is discretely compact. Arguing as in the proof of Prop. \ref{thm: difference res} it may be shown that
\begin{equation}
\label{eq:cK_n}
\cK_n(\omega) = (\curl_0 (\cC_{m,n}(\omega) - \overline{z_{\omega}})^{-1})^* (\Theta_{m,\infty}(\omega)^{-1} - \Theta_{m}(\omega)^{-1}) \curl_0 (\cC_{\infty,n}(\omega) - z_{\omega})^{-1}
\end{equation}
As a consequence of proof of Thm.\ref{thm:gsr}(iv),
\[
\begin{aligned}
(\curl_0 (\cC_{m,n}(\omega) - \overline{z_{\omega}})^{-1})^*P_n &\sto (\curl_0 (\cC_{m}(\omega) - \overline{z_{\omega}})^{-1})^* \\
\curl_0 (\cC_{\infty,n}(\omega) - z_{\omega})^{-1}P_n &\sto \curl_0 (\cC_{\infty}(\omega) - z_{\omega})^{-1},
\end{aligned}
\]
for suitably chosen $\omega$. Moreover, due to Remark \ref{rem: unif_bound_Cinfty} there exists $C > 0$ such that, for $\omega \notin \overline{W(\cC_{\infty})}$,
\[
\sup_{n \in \N}\norma{\curl_0 (\cC_{\infty,n}(\omega) - z_{\omega})^{-1}}_{\cB(H(\Div0,\Omega_n); H(\curl, \Omega_n))} \leq C,
\]
hence $\curl_0 (\cC_{\infty,n}(\omega) - z_{\omega})^{-1} u_n$ is uniformly bounded in $H(\curl, \Omega_n)$ for every sequence $u_n \in L^2(\Omega_n)^3$, $\norma{u_n} \leq 1$.
Therefore, by \eqref{eq:cK_n}, the discrete compactness of $\cK_n(\omega)$ boils down to the discrete compactness of $(\Theta_{m,\infty}(\omega)^{-1} - \Theta_{m}(\omega)^{-1})P^n_{\ker(\Div)}$ from $H(\curl, \Omega_n)$ to $L^2(\Omega_n)^3$. The proof of this last property is identical to the proof of the discrete compactness of the sequence $(\Theta_{e,\infty}(\omega) - \Theta_{e}(\omega))P^n_{\ker(\Div)}$, which was established above.
Altogether we have
\[
\sigma_e((\cT_n)_n) \cap \Sigma = (\sigma_e((-P^n_\nabla V_e(\omega)P^n_\nabla)_n)_n \cup \sigma_e((\cS_{\infty, n})_n))  \cap \Sigma,
\]
and the result follows by recalling that $G_n(\omega) = -P^n_\nabla V_e(\omega)P^n_\nabla$.
\end{proof}

We recall the following standard result, a proof of which can be found in \cite[Lemma 7.4]{BFMT}.
\begin{lemma}
\label{lemma: core}
Let $n \in \N$. The closure of $C^{\infty}_c(\Omega_{n} )^3 \cap H(\Div 0, \Omega_{n})$ with respect to the $H(\curl, \Omega_{n})$-norm is $H_0(\curl, \Omega_{n}) \cap H(\Div 0, \Omega_{n})$.
\end{lemma}

\begin{theorem} \label{thm: final} The following enclosures hold:
\begin{equation}\label{mme}
\sigma_{\rm poll}((\cL_n)_n) \cap \Sigma\subset \sigma_e((\cL_n)_n) \cap \Sigma \subset \big( W_e(\cS_{\infty}) \cup \sigma_e(G) \big) \cap \Sigma.
\end{equation}
and therefore  $(\sigma_{\rm poll}((\cL_n)_n) \cap \Sigma) \subset (W_e(\cS_{\infty}) \cap \Sigma)$. For every isolated $\omega \in (\sigma_p(\cL) \cap \Sigma)$ outside $W_e(\cS_\infty) \cup \sigma_e(G)$ there exists a sequence $\omega_n \in \sigma(\cL_n)$, $n\in\N$, such that $\omega_n \to \omega$ as $n \to \infty$.
\end{theorem}
\begin{proof}
The enclosure of spectral pollution in the limiting essential spectrum follows from Prop. \ref{prop: pollution Ln} and Corollary \ref{cor: sigmae An}. For the enclosure (\ref{mme}) we argue as in \cite[Theorem 7.5]{BFMT}. If $\omega \in \sigma_e((\cS_{\infty,n})_{n\in\N})$, by definition there exist $w_n \in \dom \cS_{\infty,n}(\omega) \subset
H_0(\curl, \Omega_n)$ $\cap H(\Div0, \Omega_n)$, $\norma{w_n} = 1$, $n\in\N$, such that $w_n \rightharpoonup 0$ and
$\cS_{\infty,n}(\omega)w_n \to 0$ as $n\to\infty$. Taking the scalar product with $w_n$, we find that
\[
\langle \cS_{\infty,n}(\omega)w_n,w_n\rangle = \Theta_{m,\infty}(\omega)^{-1} \norma{ \curl_{0} w_n}^2 - \frac{\Theta_{e,\infty}(\omega)}{(\omega + i \gamma_e)(\omega + i \gamma_m)} \to 0
\]
 as $n\to\infty$. By Lemma \ref{lemma: core}, for each $n \in \N$ there exists $v_n \in C^\infty_c(\Omega_n)^3 \cap H(\Div 0, \Omega_n)$ with $\norma{v_n-w_n}^2 \leq 1/n$, $\norma{\curl(v_n - w_n)}^2 \leq 1/n$.
Let $v_n^{0} \in H_0(\curl, \Omega) \cap H(\Div 0, \Omega)$ be the extension of $v_n$ to $\Omega$ by zero for $n\in\N$. \vspace{-1mm} Then
\begin{align*}
 & \left|\Theta_{m,\infty}(\omega)^{-1} \norma{\curl v_n^{0}}^2 - \frac{\Theta_{e,\infty}(\omega)}{(\omega + i \gamma_e)(\omega + i \gamma_m)}\norma{v_n^{0}}^2\right| \\
	&\leq \left| \Theta_{m,\infty}(\omega)^{-1}\norma{\curl w_n}^2 -  \frac{\Theta_{e,\infty}(\omega)}{(\omega + i \gamma_e)(\omega + i \gamma_m)} \norma{w_n}^2 \right| \\
& \hspace{3cm} + \frac{1}{n} \, \bigg(|\Theta_{m,\infty}(\omega)^{-1}| +  \left| \frac{\Theta_{e,\infty}(\omega)}{(\omega + i \gamma_e)(\omega + i \gamma_m)} \right| \bigg) \to 0
\end{align*}
 as $n\to\infty$.	Since $\norma{v_n^{0}} \to 1$ as $n \to \infty$, upon renormalisation of the elements $v_n^{0}$, we obtain $\omega \in W_e(\cS_{\infty})$.\\
Next, we prove the inclusion $\sigma_e((P^n_\nabla G_n(\cdot)|_{\nabla \dot H^1_0(\Omega_n)})_{n\in\N}) \subset \sigma_e(G)$.  If $\omega$ lies in $\sigma_e((P_\nabla G_n(\cdot) |_{\nabla \dot H^1_0(\Omega)})_{n\in\N})$, there exist $u_n \in \dot H^1_0(\Omega_n)$, $\norma{\nabla u_n} = 1$, $n\in\N$, such that $\nabla u_n \rightharpoonup 0$ \vspace{-1mm} and
\[
     \norma{P_{\nabla \dot H^1_0(\Omega_n)} \Theta_e(\omega)^{-1} \nabla u_n} \to 0, \quad n \to \infty.
\]
Let $u_n^{0} \in \dot H^1_0(\Omega)$ be the extension of $u_n \in \dot H^1_0(\Omega_n)$ to $\Omega$ by zero for $n\in\N$. By standard properties of Sobolev spaces, $\nabla u_n^{0} = (\nabla u_n)^{0}$. Hence the sequence ${(u_n^{0})_{n\in\N}} \subset \dot H^1_0(\Omega)$ is such that $\norma{\nabla u_n^{0}} = 1$, $n\in\N$ , $\nabla u_n^{0} \rightharpoonup 0$ and
\[
   \norma{P_{\nabla \dot H^1_0(\Omega_n)} \Theta_e(\omega)^{-1}  \nabla u_n^{0}} \to 0, \quad n \to \infty.
\]
Now the claim follows if we observe that $P_\nabla := P_{\nabla \dot H^1_0(\Omega)} f = P_{\nabla \dot H^1_0(\Omega_n)} f$ for all $f \in  L^2(\Omega)^3$ with $\supp f \subset \Omega_n$.\\
Finally, we consider the approximation of isolated eigenvalues which lie outside $W_e(\cS_\infty) \cup \sigma_e(G)$ but inside
$\Sigma$. Note first that $\sigma(\cA_n) \setminus \{-i \gamma_e, -i \gamma_m\} = \sigma(\cL_n) \setminus \{- i \gamma_e, -i \gamma_m \}$, $n \in \N$. Therefore \cite[Theorem 2.3]{MR3831156}, applied to the sequence $(\cA_n)_n$ approximating $\cA$, yields that for every isolated $\omega \in \sigma(\cA)$ outside $\sigma_e((\cA_n)_n) \cup \sigma_e((\cA^*_n)_n)^* = \sigma_e((\cA_n)_n)$ there exists $\omega_n \in \sigma(\cA_n)$, $n \in \N$, and $\omega_n \to \omega$. Since we have already proved that $\sigma_e((\cA_n)_n) \setminus \{-i \gamma_e, -i \gamma_m\} = \sigma_e((\cL_n)_n) \setminus \{-i \gamma_e, -i \gamma_m\}$ and that $\sigma_e((\cL_n)_n) \cap \Sigma \subset (W_e(\cS_\infty) \cup \sigma_e(G)) \cap \Sigma$, we deduce that every isolated point in $\omega \in \sigma(\cA) \cap \Sigma = \sigma(\cL) \cap \Sigma$, outside $W_e(\cS_\infty) \cup \sigma_e(G)$ can be approximated by spectral points $\omega_n \in \sigma(\cA_n)\setminus \{-i \gamma_e, -i \gamma_m\} = \sigma(\cL_n)$, concluding the proof.
\end{proof}

\section{Example}\label{sec:example}
We consider the Drude-Lorentz model of a dispersive metamaterial in a cuboid $K =(0,1) \times (0,L_2) \times (0,L_3)$, embedded in
an infinite waveguide $\Omega = (0, +\infty) \times (0, L_2) \times (0, L_3)$, for some $L_2, L_3 > 0$; the region $x_1>1$ in $\Omega$ is assumed to
be a vacuum. We model the discontinuity between the vacuum and the metamaterial as a discontinuity in the Drude-Lorentz parameters $\theta_e$ and $\theta_m$; note that when $\theta_e = \theta_m = 0$ we have the standard time-harmonic Maxwell system in the vacuum with permeability and permittivity constant and equal to 1. Specifically, we set
\[
\theta_e^2(x) = \alpha_e \chi_K(x), \quad \theta_m^2(x) = \alpha_m \chi_K(x), \quad \gamma_e = t > 0, \: \gamma_m = 1,
\]
where $\alpha_e, \alpha_m$ are positive constants. This leads to the coupled pair of operators
\[
\cL_1(\omega) =  \begin{pmatrix}
- \omega & i \curl \\
-i \curl_0 & - \omega
\end{pmatrix} \qquad
\cL_2(\omega) =
\begin{pmatrix}
- \omega + \frac{\alpha_e}{\omega + it} & i \curl \\
- i \curl_0 & - \omega + \frac{\alpha_m}{\omega + i}
\end{pmatrix}
\]
where $\omega \in \C \setminus \{-i, -i t\}$, $\cL_i(\omega)$ acts in $L^2(\Omega_i)^3 \oplus L^2(\Omega_i)^3$, $i=1,2$, and
\[\Omega_1 = K, \quad  \Omega_2 = (1, +\infty) \times (0,L_2) \times (0,L_3).\]
According to our results it is convenient to consider the associated first Schur complements, from which we deduce that if $(E,H)$ is an eigenfunction with eigenvalue $\omega$ then
\[
\begin{cases}
\curl \curl_0 E - \omega^2 E = 0 \quad &\textup{in $\Omega_1$,}\\
\nu \times E = 0 \quad &\textup{on $\p \Omega_1 \cap \p \Omega$,}
\end{cases}
\]
\[
\hspace{0.3cm}\begin{cases}
\curl \curl_0 E - f(\omega) E = 0 \quad &\textup{in $\Omega_2$,} \\
\nu \times E = 0 \quad &\textup{on $\p \Omega_2 \cap \p \Omega$,}
\end{cases}
\]
in which $f(\omega)$ is defined by
\[
f(\omega) = \frac{(\omega^2 + i \omega t - \alpha_e)(\omega^2 + i \omega - \alpha_m)}{(\omega + i) (\omega + it)}
\]
for all $\omega \notin \{-i, - it\}$. Note that the points $\omega = 0$ and the roots of $(\omega^2 + i \omega t - \alpha_e)(\omega^2 + i \omega - \alpha_m)$ will be in the essential spectrum, since any gradient field compactly supported in $\Omega_{i}$, $i=1,2$ will solve both systems. Define,
for $n=(n_2,n_3)\in\N^2$,
\[ \la^1_n(\omega) = \sqrt{\frac{\pi^2 n_2}{L_2^2} + \frac{\pi^2 n_3^2}{L_3^2} - f(\omega)}, \quad\quad  \la^2_n(\omega) = \sqrt{\frac{\pi^2 n_2}{L_2^2} + \frac{\pi^2 n_3^2}{L_3^2} - \omega^2}.
\]
The compatibility condition $\nu \times \curl E|_{x_1 = 0^-} = -\nu \times \curl E|_{x_1 = 0^+}$ implies that every eigenvalue $\omega$ must satisfy,
for some $n \in \N^2$, the equation
\begin{equation}\label{compat1}
\la^1_n(\omega)  \coth(\la^1_n(\omega)) + \la^2_n(\omega) = 0.
\end{equation}
Consider now the truncated domains $\Omega_X = (0,X) \times (0, L_2) \times (0, L_3)$.
The compatibility condition \eqref{compat1} now becomes
\begin{equation}\label{compat2}
\la^1_n(\omega)  \coth(\la^1_n(\omega)) + \la^2_n(\omega) \coth (\la^2_n(\omega) (X-1)) = 0.
\end{equation}
Equations (\ref{compat1},\ref{compat2}) can be solved with a standard computational engine. For the computations, we set $t=4$, $\alpha_e = 400$, $\alpha_m = 10$, see Figure \ref{fig:truncated_waveguide}. For this example, Theorem \ref{thm: final} implies that spectral pollution can only happen in $W_e(\cS_\infty)$. Since $(\omega + i \gamma_m(\omega))\cS_\infty(\omega) = \curl \curl_0 - \omega^2$ acting on divergence-free vector fields, $$W_e(\cS_\infty) = -({\rm conv}(\sigma_e(\curl\curl_0)))^{1/2} \cup ({\rm conv}(\sigma_e(\curl\curl_0)))^{1/2}.$$
Due to the divergence-free condition, $\curl \curl_0 = - \Delta$ as differential expressions. We can now perform a standard principal symbol analysis to obtain $\sigma_e(\curl\curl_0) = [(\pi/\max\{L_2, L_3\})^2, + \infty)$, and hence
\[ \sigma_e(\cS_\infty) = W_e(\cS_\infty) = \left(-\infty,\frac{-\pi}{\max\{L_2, L_3\}}\right]\cup\left[\frac{\pi}{\max\{L_2, L_3\}},+\infty\right). \]
\begin{figure}[tbh]
\centering
\includegraphics[width=0.6\textwidth]{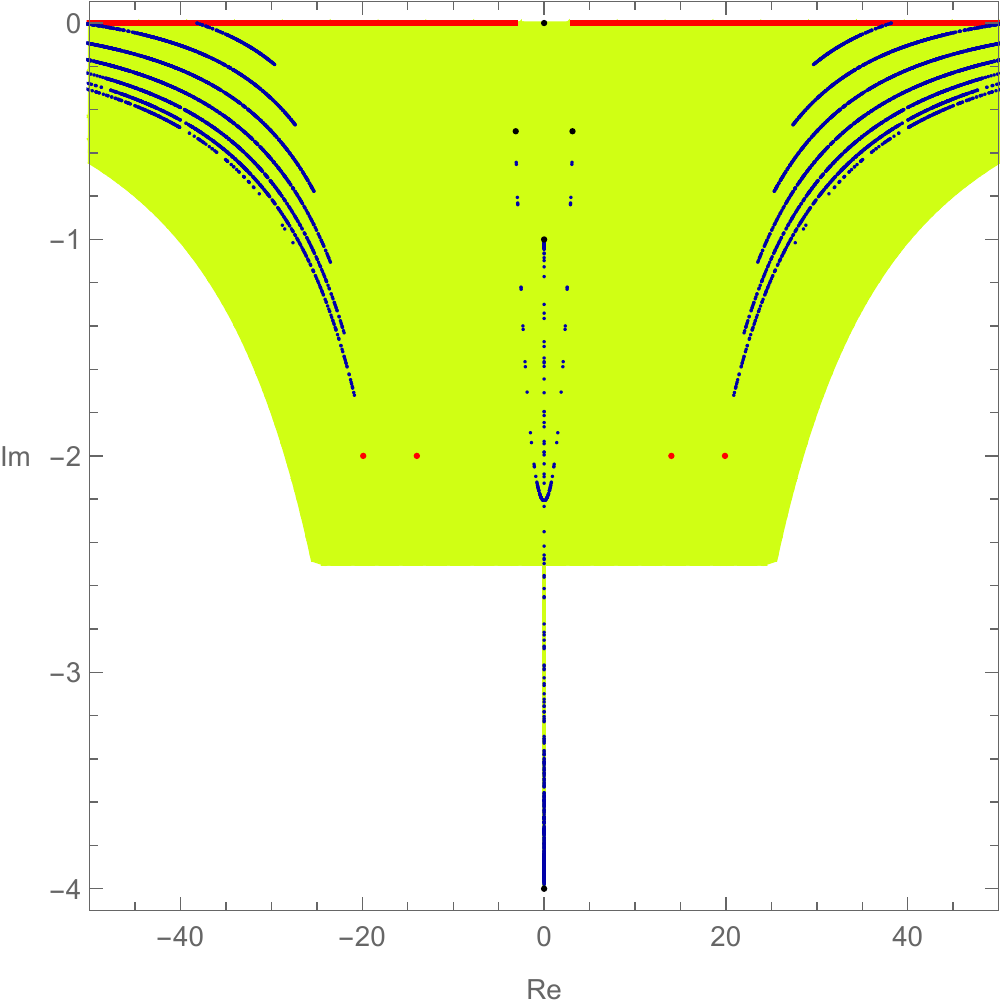}
\caption{Spectrum of the Drude-Lorentz model in the waveguide $\Omega = (0, +\infty) \times (0, 1) \times (0, \pi)$. The eigenvalues are in blue, the essential spectrum in red, the poles in black, and the spectral enclosure $\Gamma$ of Theorem \ref{thm:refnumran} in green.}
\label{fig:full_waveguide}
\end{figure}
\begin{figure}[tbh]
\centering
\label{fig:truncated_waveguide}
\includegraphics[width=0.6\textwidth]{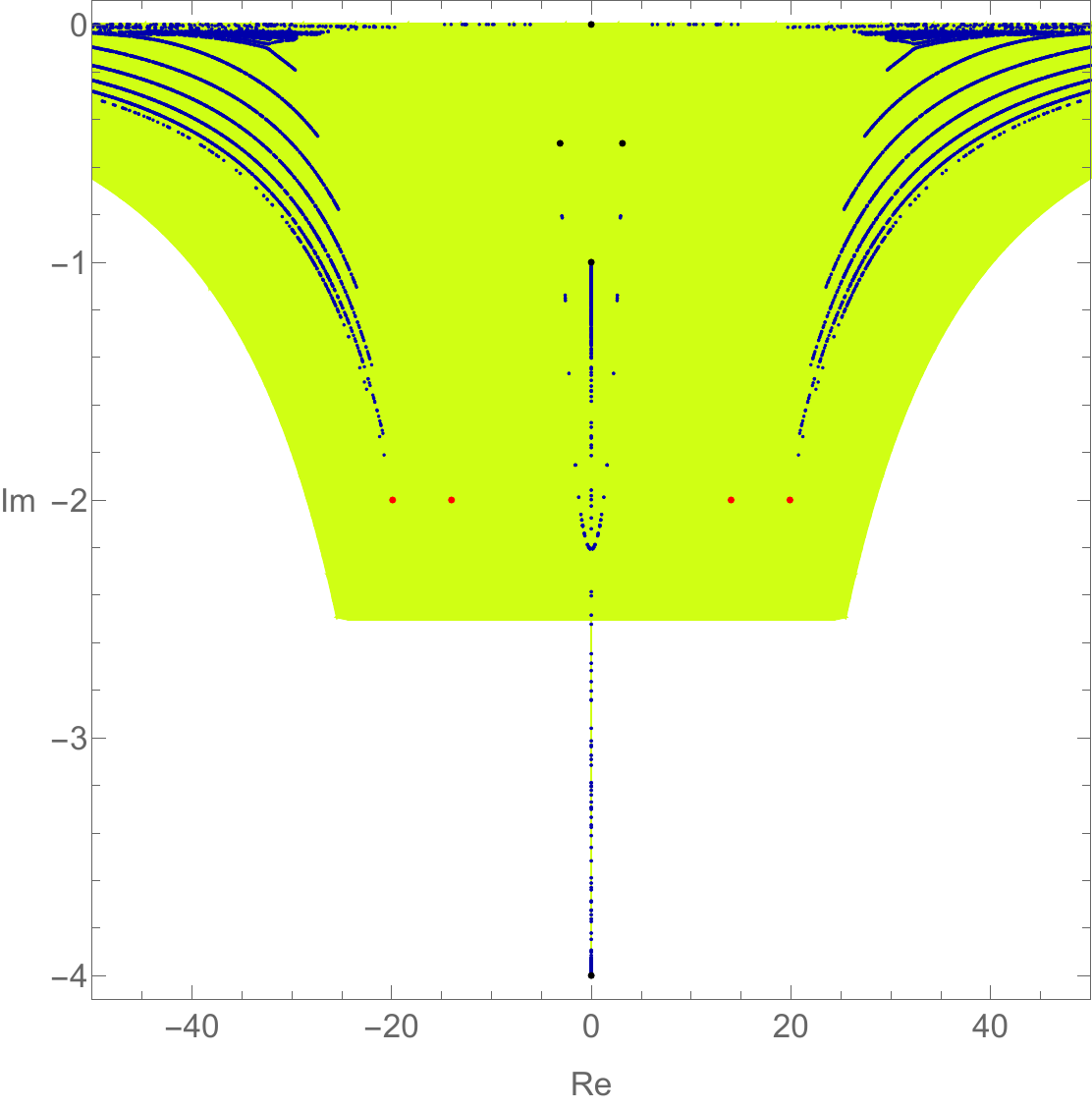}
\caption{Spectrum of the Drude-Lorentz model in the truncated waveguide $\Omega_X = (0, X) \times (0, 1) \times (0, \pi)$, with $X=25$. Accumulation of eigenvalues to the real axis is clearly visible.}
\end{figure}
Regarding $\sigma_e(G)$, let us set $(\omega^2 + \omega i \gamma_e)/ \alpha_e =: z$ and let $\omega \in \sigma_e(G)$ with associated Weyl sequence $\nabla \varphi_n \rightharpoonup 0$ in $L^2(\Omega)^3$, $\varphi_n \in \dot{H}^1_0(\Omega)$, $n \in \N$. We then have
$-P_\nabla (z - \chi_K) \nabla \varphi_n \to 0$ in $L^2(\Omega)^3$ if and only if $- \nabla (-\Delta_{\dot{H}^1_0})^{-1} \Div (z - \chi_K) \nabla \varphi_n \to 0$ in $L^2(\Omega)^3$, where $-\Delta_{\dot{H}^1_0}$ is the Dirichlet laplacian mapping $\dot{H}^1_0(\Omega)$ to its dual $\dot{H}^{-1}(\Omega)$. Now, note that $- \nabla (-\Delta_{\dot{H}^1_0})^{-1} \Div (z - \chi_K) \nabla \varphi_n \to 0$ in $L^2(\Omega)^3$ if and only if $- \Div (z - \chi_K) \nabla \varphi_n \to 0$ in $\dot{H}^{-1}(\Omega)$; the `only if' part follows immediately by applying $\nabla (-\Delta_{\dot{H}^1_0})^{-1}$, while the `if' part follows from definition of $\dot{H}^{-1}(\Omega)$.  Therefore $\sigma_e(G)$ is completely determined by $\sigma_e(- \Div (\cdot - \chi_K) \nabla)$ where for every $z \in \C$, $- \Div (z - \chi_K) \nabla$ is understood as an operator from $\dot{H}^1_0(\Omega)$ to $\dot{H}^{-1}(\Omega)$. For smooth boundaries, the problem of finding the essential spectrum of such $\Div (p(\cdot) - \chi_K) \nabla$ pencils has been recently investigated in \cite{MR4041099}. It is not too difficult to prove that $z = 0$ and $z = 1$ are in the essential spectrum of $G$. They correspond to the solutions of the two quadratic equations $\omega^2 + \omega i \gamma_e - \alpha_e = 0$ and $\omega^2 + \omega i \gamma_e = 0$. However, there are further points in the essential spectrum corresponding to $z = 1/2$. In total, therefore, $\sigma_e(G)$ consists of the six points
\[ \sigma_e(G) = \left\{0,-i\gamma_e,-i\frac{\gamma_e}{2}\pm\sqrt{\alpha_e-\frac{\gamma_e^2}{4}},-i\frac{\gamma_e}{2}\pm\sqrt{\frac{\alpha_e}{2}-\frac{\gamma_e^2}{4}}\right\}; \]
for the values used in the numerical experiments, namely $\alpha_e = 400$ and $\gamma_e = 4$, only $0$ and $-i\gamma_e$ are purely imaginary.
The four points of $\sigma_e(G)$ lying off the imaginary axis are marked in red in Fig. \ref{fig:full_waveguide}.
Moreover, we claim that the eigenvalues of $\cL$ (in blue in Figure \ref{fig:full_waveguide}) are isolated (and of finite geometric multiplicity), and therefore Theorem \ref{thm: final} implies that they are approximated without spectral pollution via domain truncation. For the claim, note that $\sigma_{e1}(\cL) = \sigma_{e2}(\cL)$ due to $\cJ$-self-adjointness of $\cL$. Also, it was proved above that $\sigma_{e1}(\cL)$ is contained in the union of two real half-lines and six isolated points. Therefore, $\Delta_{e1}(\cL) := \C \setminus \sigma_{e1}(\cL)$ has only one connected component, which has non-trivial intersection with $\rho(\cL)$. According to the notation of \cite[Chp. IX]{EE}, $\Delta_{e1}(\cL) = \Delta_{e5}(\cL)$, where $\Delta_{e5}(\cL) = \C \setminus \sigma_{e5}(\cL)$; \cite[Theorem 1.5]{EE} now implies that any $\omega \notin \sigma_{e}(\cL) = \sigma_{e5}(\cL)$ is an isolated eigenvalue (of finite geometric multiplicity). The claim is proved.
\vspace{3mm}

\noindent
{\small
{\bf Acknowledgements.}
The authors are thankful for the support of the UK Engineering and Physical Sciences Research Council through grant EP/T000902/1, \textit{`A new paradigm for spectral localisation of operator pencils and analytic operator-valued functions'}. The authors thank the anonymous referees for valuable remarks and suggestions.
}

\bibliographystyle{abbrv}
\bibliography{Maxbib}

\end{document}